\documentclass[twoside,final]{amsart}

\usepackage[utf8]{inputenc}
\usepackage[T1]{fontenc}
\usepackage[english]{babel}

\usepackage{mathpazo}
\usepackage{mathptmx}
\usepackage{amsmath}
\usepackage{amssymb}
\usepackage{mathrsfs}
\usepackage{amsthm}
\usepackage{amsfonts}
\usepackage{fancyhdr}
\usepackage[all]{xy}
\usepackage{cancel}
\usepackage{amscd,amsxtra}
\usepackage{mathtools,mathrsfs,dsfont,xparse}
\usepackage{graphicx,epstopdf}
\usepackage{multirow}
\usepackage{cases}
\usepackage{enumerate}
\usepackage{multicol,bbm}

\usepackage[notcite, notref]{showkeys}

\usepackage{mathtools}

\mathtoolsset{showonlyrefs}

\usepackage[margin=1.08 in]{geometry}

\date{}

\usepackage[colorlinks=true, pdfstartview=FitV, linkcolor=blue,
            citecolor=blue, urlcolor=blue]{hyperref}

\allowdisplaybreaks[4]

\numberwithin{equation}{section}

\newtheorem{Ass}{Assumption}

\newtheorem*{Thm*}{Theorem}

\usepackage{cjhebrew}

\newcommand{\E}{\mathbb{E}}

\newcommand{\al}{\alpha}

\newcommand{\T}{T}

\usepackage[numbers,sort&compress]{natbib}

{\theoremstyle {definition} \newtheorem {defi} {Definition} [section] }
{\theoremstyle {plain}  \newtheorem {Def} [defi] {Definition}}
{\theoremstyle {plain}  \newtheorem {Thm} [defi] {Theorem}}
{\theoremstyle {plain}  \newtheorem {Cor} [defi]{Corollary}}
{\theoremstyle {plain} \newtheorem {Prop} [defi]{Proposition}}
{\theoremstyle {plain} \newtheorem {Lem}[defi] {Lemma}}
{\theoremstyle {plain} }
{\theoremstyle {definition} \newtheorem {Rmk}[defi] {Remark}}
{\theoremstyle {definition} \newtheorem {claim}[defi] {Claim}}

{\theoremstyle {plain}  }
{\theoremstyle {plain}  }
{\theoremstyle {plain}  }
{\theoremstyle {plain} }
{\theoremstyle {plain} }
{\theoremstyle {plain} }
{\theoremstyle {plain} }

\def\E{{\Bbb{E}}}
\def\T{{\Bbb{T}}}
\def\P{{\Bbb{P}}}
\def\R{{\Bbb{R}}}
\def\C{{\Bbb{C}}}

\def\N{{\Bbb{N}}}
\def\n{{\N\backslash \{0\}}}

\def\i{{\textbf{i}}}

\def\dt{{\partial_t}}

\def\ggeq{{\ \gtrsim\ \ }}
\def\lleq{{\ \lesssim\ \ }}

\usepackage{tikz}
\usetikzlibrary{matrix}

\DeclareMathOperator{\re}{Re}

\DeclareDocumentCommand{\abs}{s m}{
  \operatorname{}
  \IfBooleanTF{#1}{#2}{\left|#2\right|}}

\DeclareDocumentCommand{\norm}{s m}{
  \operatorname{}
  \IfBooleanTF{#1}{#2} {\left\| #2\right\|}}

\DeclareDocumentCommand{\inner}{s m}{
  \operatorname{}
  \IfBooleanTF{#1}{#2}{\left \langle#2\right \rangle}}

\DeclareDocumentCommand{\parenthese}{s m}{
  \operatorname{}
  \IfBooleanTF{#1}{#2}{\left(#2\right)}}

\DeclareDocumentCommand{\square}{s m}{
  \operatorname{}
  \IfBooleanTF{#1} {#2}{\left[#2\right]}}

\DeclareDocumentCommand{\bracket}{s m}{
  \operatorname{}
  \IfBooleanTF{#1}{#2}{\left\{#2\right\}}}

\begin{document}

\author[Sy and Yu]{Mouhamadou Sy$^1$ and Xueying Yu$^2$}

\address{Mouhamadou Sy
\newline 
\indent Department of Mathematics, Imperial College London \indent 
\newline \indent  Huxley Building, London SW7 2AZ, United Kingdom,\indent }
\email{m.sy@imperial.ac.uk}
\thanks{$^1$ The first author's (M.S.) current address is Department of Mathematics, Imperial College London, United Kingdom. This work was written at Department of Mathematics, University of Virginia, Charlottesville, VA.}

\address{Xueying  Yu
\newline \indent Department of Mathematics, University of Washington\indent 
\newline \indent  C138 Padelford Hall Box 354350, Seattle, WA 98195,\indent }
\email{xueyingy@uw.edu}
\thanks{$^2$  The second author's (X.Y.) current address is Department of Mathematics, University of Washington, Seattle, WA. This work was written at Department of Mathematics, MIT, Cambridge, MA.}

\title[GWP for fractional NLS]{Global well-posedness and long-time behavior of the fractional NLS}

\begin{abstract}
In this paper, our discussion mainly focuses on equations with energy supercritical nonlinearities. We establish probabilistic global well-posedness (GWP) results for the cubic Schr\"odinger equation with any fractional power of the Laplacian in all dimensions. We consider both low and high regularities in the radial setting, in dimension $\geq 2$. In the high regularity result, an {\it Inviscid - Infinite dimensional (IID) limit} is employed while in the low regularity global well-posedness result, we make use of the Skorokhod representation theorem. The IID limit is presented in details as an independent approach that applies to a wide range of Hamiltonian PDEs. Moreover we discuss the adaptation to the periodic settings, in any dimension, for smooth regularities.


\end{abstract}

\maketitle

\setcounter{tocdepth}{1}
\tableofcontents

\parindent = 10pt     
\parskip = 8pt

\section{Introduction}
In this work, we consider the initial value problem of the cubic fractional nonlinear Schr\"odinger equation (FNLS)
\begin{align}
\dt u=-\i((-\Delta)^\sigma u+|u|^2u),\quad u\big|_{t=0}=u_0\label{fNLS}
\end{align}
on both the $d-$dimensional torus $\T^d$, $d\geq 1$, and the unit ball $B^d$, $d\geq 2$, supplemented with a radial assumption and a Dirichlet boundary condition
\begin{align}
u\Big|_{\partial B^d}=0.
\end{align}
Here $\sigma\in (0,1]$. When $\sigma=1$ we have the standard Schr\"odinger equation. The FNLS equation satisfies the following conservation laws referred as the mass and the energy
\begin{align}
M(u) &=\frac{1}{2}\|u\|_{L^2}^2,\\
E(u) &=\frac{1}{2}\|u\|_{\dot{H}^\sigma}^2+\frac{1}{4}\|u\|_{L^4}^4.
\end{align}
It also enjoys the scaling invariance given by
\begin{align}
v=\lambda^{\frac{\sigma}{2}}u(\lambda^\frac{1}{2}x,\lambda^\sigma t).
\end{align}
Consequently, its critical regularity is given by the Sobolev space $H^{s_c}$ where
\begin{align}
s_c=\frac{d}{2}-\sigma.
\end{align}
Data of regularity weaker than $s_c$ (resp. stronger than $s_c$) are called supercritical (resp. subcritical). The equation is called energy-supercritical (resp. energy-criticical, energy-subcritical) if $s_c>\sigma$ (resp. $s_c=\sigma$, $s_c<\sigma$); that is, $\sigma<\frac{d}{4}$ (resp. $\sigma=\frac{d}{4}$, $\sigma>\frac{d}{4}$). Since $\sigma\in (0,1)$, the cubic FNLS equations have energy-supercritical ranges in all dimensions $d\leq 3$, and is fully energy-supercritical in dimensions $d\geq 4.$ In this paper we address all the scenarios listed above. We construct global solutions by mean of an invariant measure argument and establish long-time dynamics properties.
\subsection{Motivation}
In recent decades, there has been of great interest in using fractional Laplacians to model
physical phenomena. The fractional quantum mechanics was introduced by Laskin \cite{laskin} as a generalization of the standard quantum mechanics. This generalization operates on the Feynman path integral formulation by replacing the Brownian motion with a general Levy flight. As a consequence,  one obtains fractional versions of the fundamental Schr\"odinger equation. That means the Laplace operator $(-\Delta)$ arising from the Gaussian kernel used in the standard theory is replaced by its fractional powers $(-\Delta)^\sigma$, where $0<\sigma<1$, as such operators naturally generate Levy flights.
Also, it turns out that the equation \eqref{fNLS} and its discrete versions are relevant in molecular biology as they have been proposed to describe the charge transport between base pairs in the DNA molecule where typical long range interactions occur \cite{dnadiscrete}. The continuum limit for discrete FNLS was studied rigorously first in \cite{kirkpatrick2013continuum}. See also \cite{grande2019continuum,grande2019space,hong2019strong} for recent works on the continuum limits.

The FNLS equation does not enjoy the strong dispersion estimates as the classical NLS does. The bounded domain setting naturally highlights this lack of dispersion. This makes both the well-posedness and the long-time behavior more difficult than the Euclidean setting on one hand, an the classical NLS on the other hand.   

The problem of understanding long-time behavior of dispersive equations on bounded domains is widely open. While on Euclidean spaces scattering turns out to apply to most of the defocusing contexts, on bounded domains one does not expect such scattering and there is no well-established general asymptotic theory.  It is however expected that solutions would generically exhibit weak turbulence\footnote{Here weak turbulence refers to an energy cascade from low to high frequencies as time evolves. The norms higher than the energy tend to blow up at infinity.}. See for example \cite{hani2015modified,hani2014long,guardia2014growth,
staffilani1997growth,bourgain1996growth,faou,
kuksin1997oscillations,CKSTTturb} and references therein for results in the direction of weak turbulence. A way to detect a weak turbulent behavior is to analyze the higher order Sobolev norms of the solutions and determine whether or not the quantity
\begin{align}
\limsup_{|t|\to\infty}\|u(t)\|_{H^s}\label{WT}
\end{align}
is infinite for some $s$.

In this paper, we prove probabilistic global well-posedness for \eqref{fNLS} by mean of invariant measures and discuss the long-time behavior by employing ergodic theorems. In particular, we show that a slightly different version of \eqref{WT} remains bounded almost surely with respect to non-trivial invariant measures of FNLS.

\subsection{History and related works}
Let us briefly present some works related to \eqref{fNLS}. The authors of \cite{tzi} employed a {\it high-low method}, originally due to Bourgain \cite{bourRef}, to prove global well-posedness for \eqref{fNLS} posed on the circle with $\sigma\in (\frac{1}{2},1)$ below the energy space. Local well-posedness on subcritical and critical regularities and global existence for small data was proved in \cite{fnls3}, the authors of \cite{fnls3} established also global well-posedness in the energy space for some powers $\sigma$ depending on the dimension, see also \cite{fnls2} where a convergence from the fractional Ginzburg-Landau to the FNLS was obtained. FNLS with Hartree type nonlinearity was studied in \cite{fnls1}. On the circle, Gibbs measures were constructed in \cite{sun2020gibbs} and the dynamics on full measure sets with respect to these measures was studied.
The authors of the present paper proved the deterministic global well-posedness below the energy space for FNLS posed on the unit disc by extending the I-method (introduced by Colliander, Keel, Staffilani, Takaoka and Tao \cite{CKSTTAlmost}) to the fractional context \cite{syyu2}.
\subsection{Invariant measure as a tool of globalization}
Bourgain \cite{bourg94} performed an ingenious argument based on the invariance of the Gibbs measure of the nonlinear one-dimensional NLS to prove global existence on a rich set of rough data. The Gibbs measure of the one-dimensional NLS is supported on the space $H^{\frac{1}{2}-}$ and was constructed by Lebowitz, Rose and Speer \cite{LRS}. The non existence of a conservation law at this level of regularity was a serious issue for the globalization. The method introduced by Bourgain to solve this issue was to derive individual bounds from the the statistical control given by the Fernique theorem. A crucial fact is that these bounds do not depend on the dimension of the Galerkin projections of the equation which are used in the construction of the Gibbs measure. With such estimates, the local solutions of the equation are globalized by an iteration argument based on comparison with the global solutions of the finite-dimensional projections. This is how an invariant measure can play the role of a conservation law. More precisely, we should say how the role played by a conservation law survive through the invariant measure; indeed such quantity are essentially constructed by the use of a conservation laws and can be seen as their statistical duals. This approach is widely exploited in different contexts. For instance, we refer the reader to the papers \cite{tzvNLS06,tzvNLS,bourbulNLS,btt18,dnygibbs,bourgNLS96} and references therein.

A second approach based on a fluctuation-dissipation method was introduced by Kuksin \cite{kuk_eul_lim} in the context of the two-dimensional Euler equation, and developed in the context of the cubic NLS by Kuksin and Shirikyan \cite{KS04}. This method uses a compactness argument and relies on stochastic analysis with an inviscid limit. The Hamiltonian equation is viewed as a limit as viscosity goes to 0 of an appropriately scaled fluctuation-dissipation equation that enjoys a stationary measure for any given vicosity. By compactness, this family of stationary measures admit at least one accumulation point as the viscosity vanishes, this turns out to be an invariant measure for the limiting equation. The scaling of the fluctuations with respect to the dissipation is such that one obtains estimates not depending on the viscosity. These estimates allow to perform a globalization argument. See \cite{sybo,sykg,fsSQG,KS12} for works related to this method. 

The two approaches discussed above were developed on various setting. However for energy supercritical equations, both approaches come across serious obstructions. The first author initiated a new approach that combines the two methods and applied it to the energy-supercritical NLS \cite{syNLS7}, this approach was developed further by the authors in \cite{syyu}. This approach utilizes a fluctuation-dissipation argument on the Galerkin approximations of the equation, we then have a double approximation: a finite-dimensional one and a viscous one. Two limits are considered in the following order: (i) an inviscid limit to recover the Galerkin projection of the Hamiltonian equation, (ii) then an infinite-dimensional limit to recover the equation itself. In the infinite-dimensional limit, a Bourgain type iteration is used. Overall, different difficulties arise in adapting the Bourgain strategy to a non Gaussian situation. That is why new ingredients were involved. This method is what we refer to {\it Inviscid-Infinite-dimensional limit}, or simply the  ``IID'' limit. See also the work \cite{lat} for a similar procedure. We perform in Section \ref{Sect:IID} a general and independent version of the IID limit.   
\subsection{Main results}
We set the regularity index for local well-posedness
\begin{align}\label{eq s*}
s_l (\sigma) = 
\begin{cases}
\frac{d}{2}, & \text{ if } \sigma \in (0,\frac{1}{2}),\\
\frac{d-1}{2}, & \text{ if } \sigma \in [\frac{1}{2} ,1),\\
\frac{d-2}{2}, & \text{ if } \sigma =1.
\end{cases}
\end{align}
Once we fix the dimension and the fractional powere in the FNLS equaation,  that is $d\geq 1$ and $\sigma\in(0,1]$, we define the following interval $I_g (d,\sigma)$, which is the range of globalization in Theorem \ref{asGWP:1}.
\begin{align}
   I_g:=I_g (d,\sigma) = \cup_{i=1}^3I_g^i(d,\sigma),
\end{align}
where $I_g^i:=I_g^i(d,\sigma)$, $i=1,2,3$ are defined as follows
\begin{align}\label{eq s_p}
\begin{cases}
I_g^1=(\frac{d}{2},\infty), & \text{for all $d\geq 1$ and $\sigma \in (0,1]$,}\\
I_g^2=(s_l(\sigma),1+\sigma], & \text{ if $d < 3+2\sigma$ and $\sigma\in[\frac{1}{2},1)$,}\\
I_g^3=(s_l(\sigma),2], & \text{ if $d \leq 5$ and $\sigma=1$.}
\end{cases}
\end{align}
See Remark \ref{RMK} below for a comment on the intervals $I_g$; it is a result of constraints imposed by the local well-posedness index above and statistical estimations on the dissipation of the energy in Section \ref{Sect:Pr as GWP 1}.
\begin{Thm}\label{asGWP:1}
Let $\sigma\in (0,1]$ and $d\geq 2$. Let $s\in I_g$,  and $\xi:\R\to\R$ be a one-to-one concave function. Then there is a probability measure $\mu=\mu_{\sigma,s,\xi,d}$  and a set $\Sigma=\Sigma_{\sigma,s,\xi,d}\subset H^s_{rad}(B^d)$ such that
\begin{enumerate}
\item $\mu(\Sigma)=1$;
\item The cubic FNLS is globally well-posed on $\Sigma$;
\item The induced flow $\phi_t$ leaves the measure $\mu$ invariant;
\item We have that
\begin{align}
\int_{L^2}\|u\|_{H^s}^2\mu(du)<\infty;\label{HS-Cont-1}
\end{align}
\item We have the bound
\begin{align}
\|\phi_tu_0\|_{H^{s-}}\leq C(u_0)\xi(1+\ln(1+|t|))\quad t\in\R;\label{Bound}
\end{align}
\item The set $\Sigma$ contains data of large size, namely for all $K>0$, $\mu(\{\|u\|_{H^s}>K\})>0$.
\end{enumerate}
\end{Thm}
\begin{Rmk}
Notice the strong bound obtained in \eqref{Bound}. This is to be compared with the bounds obtained in the Gibbs measures context which are of type $\sqrt{\ln(1+t)}$ (see for instance \cite{tzvNLS06,tzvNLS,tzvNLS06,bourg94}). 
\end{Rmk}

Theorem \ref{asGWP:1} is based on a deterministic local well-posedness. However, for $\sigma\leq \frac{1}{2}$ or $s\leq \frac{d-1}{2}$, we do not have such local well-posedness. We have a different result in this case. Set the probabilistic GWP interval for Theorem \ref{asGWP:2} below:
\begin{align}
    J_g =J^1_g\cup J^2_g
\end{align}
where
\begin{align}
\begin{cases}
J^1_g=(0,\sigma], & \sigma\in(1/2,1]\\
J^2_g =[\max(1/2,\sigma),1+\sigma], & \sigma\in(0,1].
\end{cases}   
\end{align}

\begin{Thm}\label{asGWP:2}
Consider $\sigma\in (0,1]$ and $s\in J_g$,  let $d\geq 2$. There is a measure $\mu=\mu_{\sigma,s,d}$ and a set $\Sigma=\Sigma_{\sigma,s,d}\subset H_{rad}^s(B^d)$ such that
\begin{enumerate}
\item $\mu(\Sigma)=1$;
\item The cubic FNLS is globally well-posed on $\Sigma$;
\item The induced flow $\phi_t$ leaves the measure $\mu$ invariant;
\item We have that
\begin{align}
\int_{L^2}\|u\|_{H^s}^2\mu(du)<\infty;\label{HS-Cont-2}
\end{align}
\item The set $\Sigma$ contains data of large size, namely for all $K>0$, $\mu(\{\|u\|_{H^s}>K\})>0$.
\end{enumerate}
\end{Thm}

A common corollary to Theorems \ref{asGWP:1} and \ref{asGWP:2} is as follows
\begin{Cor}
For any $u_0\in \Sigma$, where $\Sigma$ is any of the sets constructed in Theorems \ref{asGWP:1} and \ref{asGWP:2}, there is a sequence of times $t_k\uparrow\infty$ such that
\begin{align}
\lim_{k\to\infty}\|u_0-\phi_{ t_k}u_0\|_{H^s}=0.
\end{align}
\end{Cor}
The corollary above is a direct application of the Poincar\'e recurrence theorem and describes a valuable asymptotic property of the flow. Another corollary of interest is a consequence of the Birkhof ergodic theorem \cite{krengel2011ergodic}\footnote{Notice that we use here the  version of the Birkhof ergodic theorem (Theorem 2.3 in \cite{krengel2011ergodic}) that doesn't require ergodicity of the dynamics under the considered measure. Therefore, we then loose in the characterisation of the limits of the orbital averages. However the claim does not need such characterisation, it requires only the finiteness of the limits which holds true with the invariance property of the measure.}. From \eqref{HS-Cont-1} and \eqref{HS-Cont-2}, we have that for the data $u_0$ constructed in Theorems \ref{asGWP:1} and \ref{asGWP:2},
\begin{align}
\limsup_{T\to\infty}\frac{1}{T}\int_0^T\|\phi_t u_0\|_{H^s}^2dt<\infty.\label{MeanWT}
\end{align} 
The quantity above is slightly weaker than that given in \eqref{WT}. Even though \eqref{MeanWT} does not rule out weak turbulence for the concerned solutions, it gives to a certain extent an `upper bound' on the eventual energy cascade mechanism. The estimate \eqref{MeanWT} is especially important in the context of the data constructed in Theorem \ref{asGWP:1} where the regularity $H^s$ can be taken arbitrarily high.
\begin{Rmk}\label{RMK}
The intervals given in \eqref{eq s_p} can be explained as follows: the result of Theorem \ref{asGWP:1} requires a deterministic local well-posedness and a strong statistical estimate. This statistical estimate is obtained by using the dissipation operator $\mathcal{L}(u)$ defined in \eqref{Lu}. The operator $\mathcal{L}(u)$ is defined differently on low and high regularities. In fact the estimates in low regularities $\max(\sigma,\frac{1}{2})\leq s\leq 1+\sigma$ rely on a use of a C\'ordoba-C\'ordoba inequality (see Corollary \ref{Lem:Cor-Cor}) while the high regularity estimates $s>\frac{d}{2}$ use an algebra property. It is clear that the FNLS has a good local well-posedness on $H^s$ for $s>d/2$ for all $\sigma \in (0,1]$. This explains the globalization interval given in the first scenario in \eqref{eq s_p}. However, in low regularities the local well-posedness (LWP) is only valid for some indexes: (i) for $\sigma=1$, LWP holds for $s> \frac{d-2}{2}$, then we can globalize for $\max(\sigma,\frac{1}{2},\frac{d-2}{2})=\max(1,\frac{d-2}{2})<s\leq 1+\sigma$ which necessitates $d\leq 5$; hence the third scenario in \eqref{eq s_p}. (ii) For $\sigma \in [\frac{1}{2},1)$, LWP holds for $s>\frac{d-1}{2}$, which leads to the second scenario in \eqref{eq s_p}.\\
Let us also remark that for the classical cubic NLS ($\sigma=1$), an invariant measure was constructed in dimension $d\leq 4$ on the Sobolev space $H^2$ in \cite{KS04}. See also \cite{syNLS7} for higher dimensions and higher regularities invariant measures for the periodic classical NLS ($\sigma=1$).
\end{Rmk}

The complete description of the supports of fluctuation-dissipation measures is a very difficult open question. Only few results are known on substantially simple cases compared to nonlinear PDEs (see e.g. \cite{matpar,bed}). 
In the context of nonlinear PDEs,  it is traditional to ask about qualitative properties to exclude some trivial scenarios.  Without giving details of computation we refer to  \cite{armen_nondegcgl} and  Theorem $9.2$ and Corollary $9.3$ in \cite{syNLS7} as a justification of the following statement which is valid for the different settings presented in Theorems \ref{asGWP:1} and \ref{asGWP:2}. 
\begin{Thm}
The distributions via $\mu$ of the functionals $M(u)$ and $E(u)$ have densities with respect to the Lebesgue measure on $\Bbb R.$
\end{Thm}
\begin{Rmk}
The equation \eqref{fNLS} admits planar waves as solutions, we cannot exclude the scenario that the measures constructed here are concentrated on the set of these trivial solutions. However, this property is very unlikely because of the scaling between the dissipation and the fluctuations which leaves a balance in the stochastic equation \eqref{SNLS1}: If the inviscid measures were concentrated on planar waves, the latter should be very close to the "complicated" solutions of \eqref{SNLS1} (these are not trivial at all because of the noise and the scaling); for small $\alpha$ and large $N$, this would result in a non-trivial and highly surprising attractivity of the planar waves in the dynamics of \eqref{fNLS}.
\end{Rmk}
\subsection{Adaptation to the periodic case}
\begin{Thm}
\begin{enumerate}
\item Let $d\geq 1$, $\sigma\in (0,1]$ and $s>\frac{d}{2}$, the result of \label{Adapt:1} Theorem \ref{asGWP:1} are valid on $\T^d$.
\item The result of Theorem \ref{asGWP:2} are valid on $\T^d$ for:\label{Adaapt:2}
\begin{enumerate}
\item $d=1,2$ and $s\in[\max(1/2,\sigma),1]$ for all $\sigma\in(0,1]$;
\item $d=3$ and $s\in[\max(1/2,\sigma),1]$ for all $\sigma\in(1/2,1]$;
\end{enumerate}
\end{enumerate}
\end{Thm}

\begin{Rmk}
\begin{enumerate}
\item 
For the periodic extension of Theorem \ref{asGWP:1} given in item \eqref{Adapt:1}, we just notice that the smoothness of the regularities ($s>d/2$) allows a naive fixed point argument (without using Strichartz estimates) for a local well-posedness result and the calculation in Section \ref{Sect:Pr as GWP 1} can be used to obtain the exponential bounded needed in the Bourgain argument.

\item
The argument behind the extensions in item \eqref{Adaapt:2} above is essentially the fact that the uniqueness argument in Theorem \ref{asGWP:2} is based on  the radial Sobolev inequality which requires an $H^{\frac{1}{2}+}$ regularity. This regularity however is not enough for the periodic case if $d\geq 2.$ Nevertheless, for $d=1,2$ and $\frac{d}{2}<s\leq 1+\sigma$, we have uniqueness in $H^s(\T^d)$, thanks to the embedding $H^s\subset L^\infty$, and the same procedure gives the claim for $d=3.$ However, once $d\geq 4$, Theorem \ref{asGWP:2} fails on $\T^d$ because $\frac{d}{2}\geq 1+\sigma$ for all $\sigma\in (0,1]$.
\end{enumerate}
\end{Rmk}
\subsection{Comparison with \cite{syNLS7,syyu}}
The probabilistic technique employed in this paper is closed to \cite{syNLS7,syyu}. It is worth mentioning the novelty in the present work beside the fact that fractional NLS equations are much less understood than the standard NLS which motivated our interest to the problem considered here, and the presentation of the IID limit in a general form that is independent of characteristics of \eqref{fNLS}. A main difference 
from \cite{syNLS7,syyu} is the fact that the result of Theorem \ref{asGWP:2} goes below the energy space, and  $s$ can go all the way toward $0$. The major issue in achieving this is the uniqueness in low regularity. Our strategy is to control the gradient of the nonlinearity, that is the term $|u|^2$, and then to combine it with the radial Sobolev embedding and an approximation argument. This control is anticipated in the preparation of the dissipation operator.\\
Let us present the dissipation operator (see \eqref{Lu}) on which the low regularity result (see e.g.  Section \ref{Sect:Pr as GWP2}) is based, in particular:
\begin{align*}
\mathcal{L}(u)=e^{\xi^{-1}(\|u\|_{H^{s-}})}\left(\Pi^N|u|^2u+(-\Delta)^{s-\sigma}u\right)\quad 0< s\leq 1+\sigma,\ \ \sigma\in(0,1].
\end{align*}
The dissipation rate $\mathcal{E}(u)$ of the energy will be of regularity $s$ which, for $s<\sigma$,  is weaker than the energy.  Since $\mathcal{E}(u)$ is the highest regular quantity controlled, we are then in new ranges of regularity compared to \cite{syNLS7,syyu}. To deal with these ranges we need new inputs at different levels of the proof. For instance, the large data argument relies on the identity \eqref{Est:M} whose proof cannot be achieved by using the approach of \cite{syNLS7} and \cite{syyu} (where the dissipation was of positive order and then the dissipation rate of the energy was smoother than the energy itself which made it useful in the derivation of the identity concerning the dissipation rate of the mass). We,  instead, introduce a modified approach which use a careful cutoff on the frequencies.

It is worth mentioning the central quantity that we manage to control at the finite-dimensional level (the control is however uniform in the dimension) by the use of a fluctuation-dissipation strategy:
\begin{align}
e^{\xi^{-1}(\|u\|_{H^{s-}})}\left(\||u|^2\|^2_{\dot{H}^\sigma}+\|\Pi^N|u|^2u\|_{L^2}^2+\|u\|_{H^s}^2+\||u|^2\|^2_{\dot{H}^{s-\sigma}}\right).
\end{align}
Below is the role of each term in the quantity above:
\begin{itemize}
\item The term $\|u\|_{H^s}$ determines the minimal regularity of the measure
\item The term $\||u|^2\|^2_{\dot{H}^\sigma}+\||u|^2\|^2_{\dot{H}^{s-\sigma}}$ combined with the radial Sobolev inequality and an approximation argument allows to obtain uniqueness of solutions in Theorem \ref{asGWP:2} (see Subsection \ref{Subsection:Prob}). To this end we need either $\sigma>\frac{1}{2}$ or $s-\sigma>\frac{1}{2}$, this results in a control of gradient of the nonlinearity $|u|^2u$ far from the origin of the ball, and then an approximation argument introduce in our previous work \cite{syyu} is employed.
\item The term $\|\Pi^N|u|^2u\|_{L^2}^2$ combined with the Skorokhod representation theorem and a compactness argument allows to pass to the limit $N\to\infty$ and prove the existence of solutions.
\item The term $e^{\xi^{-1}(\|u\|_{H^{s-}})}$ is employed in the Bourgain argument to obtain, in particular,  the concave bounds on the solutions claimed in Theorem \ref{asGWP:1}. 
\end{itemize}

\subsection{Organization of the paper}
We present the {\it inviscid - infinite dimensional} (IID) limit in details in Section \ref{Sect:IID}. We consider a general Hamiltonian PDE and present the general framework of the IID limit, formulate assumptions and, based on them, prove the essential steps of the method. Section \ref{Sect:Pr as GWP 1} is devoted to fulfill the assumptions made in Section \ref{Sect:IID}. Section \ref{Sect:LWP} ends the fulfillment of the assumptions by establishing the local well-posedness one, hence the proof of Theorem \ref{asGWP:1} is complete. Section \ref{Sect:Pr as GWP2} is devoted to the proof of Theorem \ref{asGWP:2}.

\subsection*{Acknowledgement} 
X.Y. was funded in part by the Jarve Seed Fund and an AMS-Simons travel grant. Both authors are very grateful to the anonymous referees for valuable comments and suggestions.

\section{Preliminaries}\label{sec Preliminaries}
In this section we present some notations, functions spaces, properties the radial Laplacian and useful inequalities.
\subsection{Notation}
We define
\begin{align*}
\norm{f}_{L_t^q L_x^r (I \times D)} : = \square{\int_I \parenthese{\int_{D} \abs{f(t,x)}^r \, dx}^{\frac{q}{r}} dt}^{\frac{1}{q}},
\end{align*}
where $I$ is a time interval and $D$ is either the ball $B^d$ or the torus $\T^d$.

For $x\in \R$, we set $\inner{x} = (1 + \abs{x}^2)^{\frac{1}{2}}$. We adopt the usual notation that $A \lesssim  B$ or $B \gtrsim A$ to denote an estimate of the form $A \leq C B$ , for some constant $0 < C < \infty$ depending only on the {\it a priori} fixed constants of the problem. We write $A \sim B$ when both $A \lesssim  B $ and $B \lesssim A$.

For a real number $a$, we set $a-$ (resp. $a+$) to represent to numbers $a-\epsilon$ (resp. $a+\epsilon$) with $\epsilon$ small enough.

For a metric space $X$, we denote by $\mathfrak{p}(X)$ the set of probability measures on $X$ and $C_b(X)$ is the space of bounded continuous functions $f:X\to\R.$ If $X$ is a normed space, $B_R(X)$ represents the ball $\{u\in X\ |\ \|u\|_{X}\leq R\}$.

\subsection{Eigenfunctions and eigenvalues of the radial Dirichlet Laplacian on the ball}
From Section $2$ in \cite{atGRS}, one has the following bound for the eigenfunctions of the radial Laplacian  
\begin{align}\label{eq e_n bdd}
\norm{e_n}_{L_x^p(B^d)} & \lesssim 
\begin{cases}
1 , & \text{ if } 2 \leq p < \frac{2d}{d-1},\\
\ln (2+n)^{\frac{d-1}{2d}} & \text{ if } p=  \frac{2d}{d-1},\\
n^{ \frac{d-1}{2} -\frac{d}{p}} , & \text{ if } p>  \frac{2d}{d-1}.
\end{cases} 
\end{align}
We have also the asymptotics for the eigenvalues
\begin{align}
z_n\sim n. \label{eq z_n}
\end{align}
\subsection{$H_{rad}^s$ spaces}
Recall that $(e_n)_{n=1}^{\infty}$ form an orthonormal bases of the Hilbert space of $L^2$ radial functions on $B^d$. That is, 
\begin{align*}
\int e_n^2 \, dL = 1
\end{align*}
where $dL$ is the normalized Lebesgue measure on $B^d$. 
Therefore, we have the expansion formula for a function $u \in L^2 (B^d)$, 
\begin{align*}
u=\sum_{n=1}^{\infty} \inner{u , e_n} e_n .
\end{align*}
For $s \in \R$, we define the Sobolev space $H_{rad}^{s} (B^d)$ on the closed unit ball $B^d$ as 
\begin{align*}
H_{rad}^{s} (B^d) : = \bracket{ u = \sum_{n=1}^{\infty} c_n e_n, \, c_n \in \C : \norm{u}_{H^{s} (B^d)}^2 = \sum_{n=1}^{\infty} z_n^{2s} \abs{c_n}^2 < \infty } .
\end{align*}
We can equip $H_{rad}^{s} (B^d)$ with the natural complex Hilbert space structure. In particular, if $s =0$, we denote $H_{rad}^{0} (B^d)$ by $L_{rad}^2 (B^d)$. For $\gamma \in \R$, we define the map $\sqrt{-\Delta}^{\gamma}$ acting as isometry from $H_{rad}^{s} (B^d)$ and $H_{rad}^{s - \gamma} (B^d)$ by
\begin{align*}
\sqrt{-\Delta}^{\gamma} (\sum_{n=1}^{\infty} c_n e_n) = \sum_{n=1}^{\infty} z_n^{\gamma} c_n e_n .
\end{align*}
We denote $S_{\sigma}(t) = e^{- \i t (-\Delta)^{\sigma}}$ the flow of the linear Schr\"odinger equation with Dirichlet boundary conditions on the unit ball $B^d$, and it can be written as
\begin{align*}
S_{\sigma}(t)  (\sum_{n=1}^{\infty} c_n e_n) = \sum_{n=1}^{\infty} e^{-\i t z_n^{2 \sigma} } c_n e_n.
\end{align*}

\subsection{$X_{\sigma, rad}^{s,b}$ spaces}
Using again the $L^2$ orthonormal basis of eigenfunctions $\{ e_n\}_{n=1}^{\infty}$ with their eigenvalues $z_n^2$ on $B^d$, we define the $X^{s,b}$ spaces of functions on $\R \times B^d$ which are radial with respect to the second argument.
\begin{defi}[$X_{\sigma, rad}^{s,b}$ spaces]\label{defn Xsb}
For $s \geq 0$ and $b \in \R$,
\begin{align*}
X_{\sigma, rad}^{s,b} (\R \times B^d) = \{ u \in \mathcal{S}' (\R , L^2(B^d)) : \norm{u}_{X_{\sigma, rad}^{s,b} (\R \times B^d)} < \infty \} ,
\end{align*}
where 
\begin{align}\label{eq Xsb}
\norm{u}_{X_{\sigma, rad}^{s,b} (\R \times B^d)}^2 = \sum_{n=1}^{\infty} \norm{\inner{\tau + z_n^{2\sigma}}^b \inner{z_n}^{s} \widehat{c_n} (\tau) }_{L^2(\R_{\tau} ) }^2
\end{align}
and
\begin{align*}
u(t) = \sum_{n=1}^{\infty} c_n (t) e_n .
\end{align*}
Moreover, for $u \in X_{\sigma, rad}^{0, \infty} (\R \times B^d) =  \cap_{b \in \R} X_{\sigma, rad}^{0,b} (\R \times B^d)$ we define, for $s \leq 0$ and $b \in \R$, the norm $\norm{u}_{X_{\sigma, rad}^{s,b} (\R \times B^d)}$ by \eqref{eq Xsb}.
\end{defi}
Equivalently, we can write the norm \eqref{eq Xsb} in the definition above into
\begin{align*}
\norm{u}_{X_{\sigma, rad}^{s,b} (\R \times B^d)} = \norm{S_{\sigma} (-t) u}_{H_t^b H_x^s (\R \times B^d)} .
\end{align*}

For $T > 0$, we define the restriction spaces $X_{\sigma, T}^{s,b} (B^d)$ equipped with the natural norm
\begin{align*}
\norm{u}_{X_{\sigma, T}^{s,b} ( B^d)} = \inf \{ \norm{\tilde{u}}_{X_{\sigma, rad}^{s,b} (\R \times B^d)} : \tilde{u}\big|_{(-T,T) \times B^d} =u\} .
\end{align*}

\begin{Lem}[Basic properties of $X_{\sigma, rad}^{s,b}$ spaces]\label{lem X property1}
\begin{enumerate}
\item
We have the trivial nesting 
\begin{align*}
X_{\sigma, rad}^{s,b} \subset  X_{\sigma, rad}^{s' , b' } 
\end{align*}
whenever $s' \leq s$ and $b' \leq b$, and
\begin{align*}
X_{\sigma, T}^{s,b} \subset  X_{\sigma, T'}^{s,b} 
\end{align*}
whenever $T' \leq T$ .
\item
The $X_{\sigma, rad}^{s,b}$ spaces interpolate nicely in the $s, b$ indices.
\item
For $b > \frac{1}{2}$, we have the following embedding
\begin{align*}
\norm{u}_{L_t^{\infty} H_x^{s} (\R \times B^d) } \leq C \norm{u}_{X_{\sigma, rad}^{s,b} (\R \times B^d)}.
\end{align*}
\item An embedding that will be used frequently in this paper
\begin{align*}
X_{\sigma, rad}^{0, \frac{1}{4}} \hookrightarrow L_t^4 L_x^2 .
\end{align*}
\end{enumerate}
\end{Lem}

Note that 
\begin{align*}
\norm{f}_{L_t^{4} L_x^2} = \norm{S_{\sigma}(t)  f}_{L_t^{4} L_x^2} \leq \norm{S_{\sigma}(t)  f}_{H_{t,rad}^{\frac{1}{4}} L_x^2}  = \norm{f}_{X_{\sigma, rad}^{0, \frac{1}{4}}} .
\end{align*}

\begin{Lem}\label{lem X property2}
Let $b,s >0$ and $u_0 \in H_{rad}^s (B^d)$. Then there exists $c >0$ such that for $0 < T \leq 1$,
\begin{align*}
\norm{S_{\sigma}(t)  u_0}_{X_{\sigma, rad}^{s,b} ((-T, T) \times B^d)} \leq c \norm{u_0}_{H^s}.
\end{align*} 
\end{Lem}

The proofs of Lemma \ref{lem X property1} and Lemma \ref{lem X property2}  can be found in \cite{an}.

We also recall the following lemma in \cite{BourExp, gi}
\begin{Lem}\label{lem Duhamel}
Let $0 < b' < \frac{1}{2}$ and $0 < b < 1-b'$. Then for all $f \in X_{\sigma, \delta}^{s, -b'} (B^d)$, we have the Duhamel term $w(t) = \int_0^t  S_{\sigma}(t-s) f(\tau) \, ds \in X_{\sigma, \delta}^{s,b} (B^d)$ and moreover
\begin{align*}
\norm{w}_{X_{\sigma, \delta}^{s,b} (B^d)} \leq C \delta^{1 +2 b - 4b'} \norm{f}_{X_{\sigma, \delta}^{s,-b'} (B^d)} .
\end{align*} 
\end{Lem}

\noindent From now on, for simplicity of notation, we write $H^s$ and $X_{\sigma}^{s,b}$ for the spaces $H_{rad}^s$  and $X_{\sigma, rad}^{s,b}$ defined in this subsection.

\subsection{Useful inequalities}

\begin{Lem}[C\'ordoba-C\'ordoba inequality\cite{cordoba2003pointwise,constantin2017remarks}]\label{Lem:Cor-Cor}
Let $D\subset\R^n$ be a bounded domain with smooth boundary (resp. the $n$-dimensional torus). Let $\Delta$ be the Laplace operator on $D$ with Dirichlet boundary condition (rep. periodic condition). Let $\Phi$ be a convex $C^2(\R,\R)$ satisfying $\Phi(0)=0$, and $\gamma\in [0,1]$. For any $f\in C^\infty(D,\R)$, the inequality
\begin{align}
\Phi'(f)(-\Delta)^\gamma f\geq (-\Delta)^\gamma\Phi(f)
\end{align}
holds pointwise almost everywhere in $D.$
\end{Lem}

\begin{Lem}[Complex C\'ordoba-C\'ordoba inequality]\label{lem CC}
Let $D$ and $\Delta$ be as in Lemma \ref{Lem:Cor-Cor}. For any $f\in C^\infty(D,\C)$, the inequality
\begin{align}
2\Re [f(-\Delta)^\gamma \bar{f}]\geq (-\Delta)^\gamma|f|^2
\end{align}
holds pointwise almost everywhere in $D.$
\end{Lem}
\begin{proof}[Proof of Lemma \ref{lem CC}]
Let us write $f=a+\i b$ with real valued functions $a$ and $b$. We have
\begin{align}
\Re[(a+\i b)(-\Delta)^{\gamma}(a-\i b)]=a(-\Delta)^{\gamma}a+b(-\Delta)^{\gamma}b.
\end{align}
Now we use Lemma \ref{Lem:Cor-Cor} with $\Phi(x)=\frac{x^2}{2}$, we arrive at
\begin{align}
\Re[(a+\i b)(-\Delta)^{\gamma}(a-\i b)]\geq \frac{1}{2}(-\Delta)^\gamma(a^2+b^2)=\frac{1}{2}(-\Delta)^\gamma|f|^2.
\end{align}
This finishes the proof of Lemma \ref{lem CC}.
\end{proof}
\begin{Cor}\label{Lem:Cplx:Cor}
Let $\gamma\in[0,1]$, we have, for $f\in C^\infty(D,\C)$, that
\begin{align}
\langle(-\Delta)^{\gamma}f,|f|^{2}f\rangle\geq \frac{1}{2}\|(-\Delta)^\frac{\gamma}{2}|f|^2\|_{L^2}^2.
\end{align}
\end{Cor}
\begin{Lem}[Radial Sobolev lemma on the unit ball]\label{lem radSob}
Let $\frac{1}{2}<s<\frac{d}{2}$. Then for any $f\in H^s:=H_0^s(B^d)$\footnote{$H^s_0(B^d)$ is the Sobolev space of order $s$ of functions $f:B^d\to\C$ vanishing on the boundary of $B^d.$}, we have
\begin{align}
|f(r)|\lleq r^{s-\frac{d}{2}}\|f\|_{H^s}\quad \text{for all $r\in (0,1].$}\label{Radial:Sob}
\end{align}
\end{Lem}
\begin{proof}[Proof of Lemma \ref{lem radSob}]
Let $\frac{1}{2}<s<\frac{d}{2}$. We have for any $f\in H^s(\R^d)$ (see \cite{choozawa}), that
\begin{align}
|f(r)|\lleq r^{s-\frac{d}{2}}\|f\|_{H^s(\R^d)}\quad \text{for all $r\in (0,\infty).$}
\end{align}
On the other hand, since $B^d$ is a regular domain, we have the  extension theorem (see for instance \cite{jonwall} and reference therein): there is a bounded operator $\mathfrak{E}:H^s_0(B^d)\to H^s(\R^d)$ satisfying the following
\begin{enumerate}
\item $\mathfrak{E}f(r)=f(r)\ \forall r\in [0,1]$, for any $f\in H^s_0(B^d)$; \label{loc:rs:1}
\item (continuity)
\begin{align}
\|\mathfrak{E}f\|_{H^s(\R^d)}\leq C\|f\|_{H^s_0(B^d)}.\label{los:rs:2}
\end{align}
\end{enumerate}
Therefore using \eqref{los:rs:2}, we notice that 
\begin{align}
|\mathfrak{E}f(r)|\lleq r^{s-\frac{d}{2}}\|\mathfrak{E}f\|_{H^s(\R^d)}\lleq r^{s-\frac{d}{2}}\|f\|_{H^s_0(B^d)}\quad \text{for all $r\in (0,\infty).$}
\end{align}
Then for $r\in (0,1]$, we use the item \eqref{loc:rs:1} above to arrive at
\begin{align}
|f(r)|\lleq r^{s-\frac{d}{2}}\|f\|_{H^s}.
\end{align}
Hence we finish the proof of Lemma \ref{lem radSob}.
\end{proof}

\section{Description of the Inviscid-Infinite dimensional (IID) limit}\label{Sect:IID}
The IID limit combines a fluctuation-dissipation argument and an abstract version of the Bourgain's invariant measure globalization \cite{bourg94}.

We consider a Hamiltonian equation
\begin{align}
\dt u=JH'(u)\label{Eq:Ham:Gen}
\end{align}
where $J$ is a skew-adjoint operator on suitable spaces, $H(u)$ is the hamiltonian function, $H'$ is the derivative with respect to $u$ (or a function of $u$, for instance to its complex conjugate). We can, for simplicity, assume the following form for the Hamiltonian
\begin{align}
H(u)=E_K(u)+E_p(u),
\end{align}
where $E_K$ and $E_p$ refer  to the kinetic energy of the system (having a quadratic power) and the potential energy, respectively. In the sequel we construct a general framework allowing to apply the Bourgain's invariant measure argument to the context of a general probability measure (not necessarily Gaussian based). We then present the strategy of the construction of the required measures.

\subsection{Abstract version of Bourgain's invariant measure argument}

Let $\Pi^N$ be the projection on the $N$-dimensional space $E^N$ spanned by the first $N$ eigenfunctions $\{e_1, \cdots, e_N\}$. Consider the Galerkin projections of \eqref{Eq:Ham:Gen}
\begin{align}
\dt u=\Pi^NJH'(u).\label{Eq:Ham:N}
\end{align}
The equation \eqref{Eq:Ham:Gen} will be seen as \eqref{Eq:Ham:N} with $N=\infty$.

\begin{Ass}[Uniform local well-posedness]\label{Ass1}
The equation \eqref{Eq:Ham:N} is uniformly (in $N$) locally wellposed in the Cauchy-Hadamard sense on some Sobolev space $H^s$. And there is a function $f$ independent of $N$, such that for  any given data $u_0\in \Pi^NH^s$, the time existence $T(u_0)$ of the corresponding solution $u$ is at least $f(\|u_0\|_{H^s})$. Also, the following estimate holds
\begin{align}
\|u(t)\|_{H^s}\leq 2\|u_0\|_{H^s}\quad \text{for all $|t|\leq T(u_0)$}.
\end{align}
We denote by $\phi^N_t$ the local flow of \eqref{Eq:Ham:N}, for $N=\infty$ we set $\phi_t=\phi^\infty_t$.\\
For $N<\infty$, we assume that $\phi^N_t$ is defined globally in time.
\end{Ass}
\begin{Rmk}
For the power-like nonlinear Schr\"odinger equation, $f(x)$ is of the form $x^r$ where $r$ is related to the power of the nonlinearity.
\end{Rmk}
\begin{Ass}[Convergence]\label{Ass2}
For any $u_0\in H^s$, any sequence $(u_0^N)_N$ such that 
\begin{enumerate}
\item for any $N$, $u_0^N\in \Pi^NH^s$;
\item $(u_0^N)$ converges to $u_0$ in $H^s$;
\end{enumerate}
 then $\phi^N_tu^N$ converges $\phi_tu$ in $C_tH^{s-}.$

\end{Ass}
\begin{Ass}[Invariant measure]\label{Ass3}
For any $N$, there is a measure $\mu_N$ invariant under the projection equations \eqref{Eq:Ham:N} and satisfying
\begin{enumerate}
\item There is an increasing one-to-one function $g:\R\to\R_+$ and a function $h:\R_+\to\R_+$ such that we have the uniform bound
\begin{align}
\int_{E^N}g(\|u\|_{H^{s}})h(\|u\|_{H^{s_0}})\mu^N(du)\leq C,
\end{align} 
where $C>0$ is independent of $N$ and $s_0$ is some fixed number. We then have
\begin{align}
\int_{E^N}e^{\ln(1+ g(\|u\|_{H^{s}}))}h(\|u\|_{H^{s_0}})\mu^N(du)\leq 1+C.\label{EST:EXP}
\end{align}
We set $\tilde{g}(\|u\|_{H^{s}})=\ln(1+ g(\|u\|_{H^{s}}));$\footnote{Remark that $\tilde{g}$ is increasing and one-to-one from $\R$ to $\R$.}
\item The following limit holds
\begin{align}
\lim_{i\to\infty}e^{-2i}\sum_{j\geq 0}\frac{e^{-j}}{f\circ\tilde{g}^{-1}(2(i+j))}=0.
\end{align}
We set the number
\begin{align}
\kappa(i):=e^{-2i}\sum_{j\geq 0}\frac{e^{-j}}{f\circ\tilde{g}^{-1}(2(i+j))}.
\end{align}
\end{enumerate}
\end{Ass}
\begin{Rmk}
For the case of Gibbs measures for NLS like equations, $g(x)$ is of the form $e^{ax^2},
\ a>0$ and $h(x)$ is a constant. 
\end{Rmk}
The Prokhorov theorem (Theorem \ref{Thm:Prokh}) combined with the estimate \eqref{EST:EXP} implies the existence of a probability measure $\mu\in\mathfrak{p}(H^s)$ as an accumulation point when $N\to\infty$ for the measures $\{\mu^N\}$.
\begin{Thm}\label{Thm:Gen:as}
Under the Assumptions \ref{Ass1}, \ref{Ass2}, and \ref{Ass3}, there is a set $\Sigma\subset H^s$ such that
\begin{enumerate}
\item $\mu(\Sigma) =1$
\item The equation \eqref{Eq:Ham:Gen} is GWP on $\Sigma$
\item The measure $\mu$ is invariant under the flow $\phi_t$ induced by the GWP;
\item For every $u_0\in\Sigma$, we have that
\begin{align}
\|\phi_t u_0\|_{H^{s}}\leq C(u_0)\tilde{g}^{-1}(1+\ln(1+|t|)).
\end{align}
\end{enumerate} 
\end{Thm}

The following is devoted to the proof of Theorem \ref{Thm:Gen:as}.
\subsubsection{Individual bounds}
Fix an arbitrary (small) number $a>0$, define the set
\begin{align}
E^N_a=\left\{u\in E^N\ |\ h(\|u\|_{H^{s_0}})\geq a\right\}.\label{Def:E^N_a}
\end{align}
\begin{Prop}\label{prop 3.2}
Let $r\leq s,\ N\geq 1,\ i\geq 1$, there exists a set $\Sigma^N_{r,i}\subset E^N_a$ such that
\begin{align}
\mu^N(E_a^N\backslash\Sigma^N_{r,i})\leq Ca^{-1}\kappa(i),\label{Est:Sigma^N}
\end{align}
and for any $u_0\in \Sigma^N_{r,i}$, we have
\begin{align}
\|\phi^N_tu_0\|_{H^r}\leq 2\tilde{g}^{-1}(1+i+\ln(1+|t|))\quad\text{for all $t\in\R$.}\label{Est:Bourg:N}
\end{align}
Moreover, we have the properties
\begin{align}
\Sigma^{N_1}_{s,i}\subset\Sigma^{N_2}_{s,i}\ \text{if $N_1\leq N_2$};\quad \Sigma^{N}_{s,i_1}\subset\Sigma^{N}_{s,i_2}\ \text{if $i_1\leq i_2$};\quad \Sigma^{N}_{s_1,i}\subset\Sigma^{N}_{s_2,i}\ \text{if  $s_1\geq s_2$}.\label{PROP-SIGM_N}
\end{align}
\end{Prop}
\begin{proof}[Proof of Proposition \ref{prop 3.2}]
It suffices to consider $t>0$. We define the sets
\begin{align}
B^N_{r,i,j}:=\{u\in E_a^N\ |\ \|u\|_{H^r}\leq \tilde{g}^{-1}(2(i+j))\}.\label{Def:B^N_r,i,j}
\end{align}
Let $T_0$ be the local time existence on $B^N_{r,i,j}$. From the local well-posedness, we have  for $t\leq T_0=f(\tilde{g}^{-1}(2(i+j)))$ (see Assumption \ref{Ass1}), that
\begin{align}
\phi^N_tB^N_{r,i,j}\subset\{u\in E_a^N\ |\ \|u\|_{H^r}\leq 2\tilde{g}^{-1}(2(i+j))\}.
\end{align}
Now, let us define the set
\begin{align}
\Sigma^N_{r,i,j}:=\bigcap_{k=0}^{[\frac{e^j}{T_0}]}\phi^N_{-kT_0}(B^N_{r,i,j}).\label{Def:Sigma^N_r,i,j}
\end{align}
We use \eqref{EST:EXP} and the Chebyshev inequality to obtain the following
\begin{align}
\mu^N(E_a^N\backslash\Sigma^N_{r,i,j})=\mu^N\left(\left\{u\in E^N_a\ |\ \|u\|_{H^r}\geq \tilde{g}^{-1}(2(i+j))\right\}\right) &\lleq e^{-2(i+j)}[\frac{e^j}{T_0}]\int_{E^N_a}e^{\tilde{g}\left(\|u\|_{H^r}\right)}\mu^N(du)\\
&\lleq e^{-2(i+j)}[\frac{e^j}{T_0}]a^{-1}\int_{L^2}e^{\tilde{g}(\|u\|_{H^r})}h(\|u\|_{H^{s_0}})\mu^N(du)\\
&\lleq a^{-1}[\frac{e^j}{T_0}]e^{-2(i+j)}\\
&\lleq a^{-1}\frac{e^{-2i}e^{-j}}{f\circ\tilde{g}^{-1}(2(i+j))}.
\end{align}
Let us finally define
\begin{align}
\Sigma^N_{r,i}=\bigcap_{j\geq 0}\Sigma^N_{r,i,j}.\label{Def:Sigma^N_r,i}
\end{align}
We then obtain
\begin{align}
\mu^N(E^N_a\backslash\Sigma^N_{s,i})\leq Ca^{-1}e^{-2i}\sum_{j\geq 0}\frac{e^{-j}}{f\circ\tilde{g}^{-1}(i+j)}= Ca^{-1}\kappa(i).
\end{align}
We claim that for $u_0\in \Sigma^N_{r,i,j}$, the following estimate holds
\begin{align}
\|\phi^N_tu_0\|_{H^r}\leq 2\tilde{g}^{-1}(2(i+j))\quad\text{for all $t\leq h(j)$}.
\end{align}
Indeed we write $t=kT_0+\tau$ where $k$ is an integer in $[0,\frac{e^j}{T_0}]$ and $\tau\in [0,T_0]$. We have
\begin{align}
\phi^N_tu_0=\phi^N_{\tau+kT_0}u_0=\phi^N_\tau(\phi^N_{kT_0}u_0).
\end{align}
Now, recall that $u_0\in \phi^N_{-kT_0}(B^N_{r,i,j})$, then $u_0\in\phi^N_{-kT_0}w$ where $w\in B^N_{r,i,j}$. Then $\phi^N_tu_0=\phi^N_\tau w$.

For $t>0$, there $j\geq 0$ such that
\begin{align}
e^{j-1}\leq 1+t\leq e^j.
\end{align}
Then
\begin{align}
j\leq 1+\ln(1+t).
\end{align}
We arrive at
\begin{align}
\|\phi^N_tu_0\|_{H^r}\leq 2\tilde{g}^{-1}(1+i+\ln(1+t)).
\end{align}
The properties \eqref{PROP-SIGM_N} follow easily from the definition of the sets $\Sigma^N_{r,i}$.  Remark that the definition of $\Sigma^N_{r,i}$ does not depend on the measure $\mu^N$ but just on the functions $h$ and $g$,  as can be noticed in \eqref{Def:E^N_a}, \eqref{Def:B^N_r,i,j}, \eqref{Def:Sigma^N_r,i,j} and \eqref{Def:Sigma^N_r,i}.  If $N_1\leq N_2$, then we see trivially that $B^{N_1}_{r,i,j}\subset B^{N_2}_{r,i,j}$, it then follows from \eqref{Def:Sigma^N_r,i,j} and \eqref{Def:Sigma^N_r,i} that $\Sigma^{N_1}_{r,i}\subset \Sigma^{N_2}_{r,i}$. Using the fact that $g^{-1}$ is increasing, we obtain $\Sigma^N_{s,i_1}\subset \Sigma^N_{s,i_2}$ for $i_1\leq i_2$. Similarly, the inequality $\|u\|_{H^{s_2}}\leq \|u\|_{H^{s_1}}$ for $s_1\geq s_2$ leads to remaining inclusion.
The proof of Proposition  \ref{prop 3.2} is finished.
\end{proof}
\subsubsection{Statistical ensemble}
We want to pass to the limit $N\to\infty$ in the sets $\Sigma^N_{r,i}$. We introduce the following limiting set:
\begin{align}
\Sigma_{r,i}:=\left\{u\in H^r\ |\ \exists (u^{N_k})_k,\ \lim_{k\to\infty}u^{N_k}=u,\ \text{where $u^{N_k}\in\Sigma^{N_k}_{s,i}$}\right\}.
\end{align}
We then define the following set
\begin{align}
\Sigma_r=\bigcup_{i\geq 1}\overline{\Sigma_{r,i}}.
\end{align}
Let $(l_k)$ be a sequence such that $\lim_{k\to\infty}l_k=s$, the statistical ensemble is given by
\begin{align}
\Sigma=\bigcap_{k\in l_k}\Sigma_{k}.
\end{align}

\begin{Prop}\label{Meas:Sig}
The set $\Sigma$ constructed above is of full $\mu$-measure:
\begin{align}
\mu(\Sigma)=1.
\end{align}
\end{Prop}

\begin{proof}[Proof of Proposition \ref{Meas:Sig}]
Let us first observe that
\begin{align*}
\Sigma_{r,i}^N \subset \Sigma_{r,i},
\end{align*}
because for any $u \in \Sigma_{r,i}^N $ the constant sequence $u_N = u$ converges to $u$, and then $u \in \Sigma_{r,i}$.\\
Now by the Portmanteau theorem (see Theorem \ref{Thm:Portm}) and the use of the bound \eqref{Est:Sigma^N} we have that
\begin{align*}
\mu(\overline{\Sigma_{r,i}} ) \geq \lim_{N_k \to \infty} \mu^{N_k} (\overline{\Sigma_{r,i}}) \geq \lim_{N_k \to \infty} \mu^{N_k} (\overline{\Sigma^{N_k}_{r,i}}) \geq \lim_{N_k \to \infty} \mu^{N_k} (\Sigma_{s,i}^{N_k}) \geq \lim_{N_k \to \infty} (1 - Ca^{-1}\kappa(i))  = 1 - Ca^{-1}\kappa(i).
\end{align*}
On the other hand, from the properties \eqref{PROP-SIGM_N}, we have that $(\overline{\Sigma_{r,i}})_{i \geq 1}$ is non-decreasing.
In particular, 
\begin{align*}
\mu (\cup_{i \geq 1} \overline{\Sigma_{r,i}}) = \lim_{i \to \infty} \mu (\overline{\Sigma_{r,i}}).
\end{align*}
Therefore
\begin{align*}
\mu (\cup_{i \geq 1} \overline{\Sigma_{r,i}}) = \lim_{i \to \infty} \mu (\overline{\Sigma_{r,i}}) \geq \lim_{i \to \infty} (1 - Ca^{-1}\kappa(i))  =1.
\end{align*}
Since $\mu$ is a probability measure, we obtain
\begin{align*}
1 \geq \mu(\Sigma_r) \geq 1.
\end{align*}
This finishes the proof of Proposition \ref{Meas:Sig}.
\end{proof}

\subsubsection{Globalization}

\begin{Prop}\label{prop 3.4}
Let $r\leq s$. For any $u_0\in \Sigma_r$, there is a unique global solution $u\in C_tH^s$ such that
\begin{align}
\|\phi_t u_0\|_{H^{s}}\leq 2\tilde{g}^{-1}(1+i+\ln(1+|t|))\quad \forall t\in\R.\label{Est:Bourg}
\end{align}
\end{Prop}

\begin{proof}[Proof of Proposition \ref{prop 3.4}]
For $u_0\in\Sigma_r\cap\text{supp}(\mu)$ (recall that $\mu(\Sigma_r)=1$), we have that
\begin{enumerate}
\item $u_0\in H^s$ since $\text{supp}(\mu)\subset H^s$;
\item there is $i\geq 1$ such that $u_0\in\overline{\Sigma_{r,i}}$.
\end{enumerate}
We will consider that $u_0\in\Sigma_{r,i}$, the case $u_0\in\partial \Sigma_{r,i}$ can be obtained by a limiting argument (see \cite{syyu}). By construction of $\Sigma_{r,i}$, there is $N_k\to \infty$ such that
\begin{align}
u_0=\lim_{N_k\to\infty}u^{N_k},\quad u^{N_k}\in\Sigma^{N_k}_{r,i}.
\end{align}
Now using \eqref{Est:Bourg:N}, we have
\begin{align}
\|\phi^{N_k}_tu^{N_k}\|_{H^r}\leq 2\tilde{g}^{-1}(1+i+\ln(1+|t|)),
\end{align}
this gives in particular that, at $t=0$,
\begin{align}
\|u^{N_k}\|_{H^r}\leq 2\tilde{g}^{-1}(1+i).
\end{align}
A passage to the limit $N_k\to\infty$ shows that
\begin{align}
\|u_0\|_{H^r}\leq 2\tilde{g}^{-1}(1+i).
\end{align}
Now for any fixed $T>0$, set $b=2\tilde{g}^{-1}(1+i+\ln(1+|T|))$ and $R=b+1$, so that $u^{N_k},\ u_0\in B_R(H^r)$. Let $T_R$ be the local existence time of $B_R$. We have that $\phi_t$ and $\phi^{N_k}_tu^{N_k}$ exist for $|t|\leq T_R$. Using Assumption \ref{Ass2}, we have for $r<s$, 
\begin{align}
\|\phi_tu_0\|_{L^\infty_tH^r}\leq b\leq R.
\end{align}  
This bound tells us that after the time $T_R$ the solution stays in the same ball as the initial data. We can then iterate on intervals $[nT_R,(n+1)T_R]$ still consuming the fixed time $T$. Since this time is arbitrary, we conclude to the global existence. The estimate \eqref{Est:Bourg} follows easily a passage to the limit from \eqref{Est:Bourg:N}.
Now the proof of Proposition \ref{prop 3.4} is finished. 
\end{proof}
\begin{Rmk}
It follows from the proof above that, if $u_0\in \Sigma_r$ we have that
\begin{align}
\lim_{k\to\infty}\|\phi_tu_0-\phi_t^{N_k}u^{N_k}\|_{H^r}=0,\label{Conv:flows}
\end{align}
where $u^{N_k}\in \Sigma^{N_k}_{r,i}$ and $\lim_{k\to\infty}\|u_0-u^{N_k}\|_{H^r}=0.$
\end{Rmk}

\subsubsection{Invariance results}

\begin{Thm}\label{thm inv}
The measure $\mu$ is invariant under $\phi_t$.
\end{Thm}
\begin{proof}[Proof of Theorem \ref{thm inv}]
It follows from the Skorokhod representation theorem (Theorem \ref{Thm:Skorokhod}) the existence of  $u^N_0$ a random variable distributed by $\mu^N$ and $u_0$ a one distributed by $\mu$ such that
\begin{align}
u^N_0\to u_0\quad\text{ as $N\to\infty,\ \ $ almost surely.}
\end{align}
Combining this with the Assumption \ref{Ass2}, we then have the almost sure convergence
\begin{align}
\phi^N_tu^N_0(\omega)\to\phi_tu_0(\omega) \quad \text{on $H^{s-}$}.\label{Pointwise:conv}
\end{align} 
 For any fixed $t$, we know that the distribution $\mathcal{D}(\phi^N_tu_0^N)$ is $\mu^N$, by the invariance of $\mu^N$ under \eqref{Eq:Ham:N}. Let us look at $\phi^N_tu_0^N$ and $\phi_t u_0$ as $H^{s-}-$valued random variables. We then have the following two results:
\begin{enumerate}
\item The pointwise convergence in \eqref{Pointwise:conv} implies that $\mu^N$  weakly converges to $\mathcal{D}(\phi_tu_0)$ for any fixed $t$, hence the invariance of $\mathcal{D}(\phi_tu_0)$ because $\mu^N$ does not depend on $t$;
\item We already know that  $\mu^N$ converges  weakly to $\mu$.
\end{enumerate}
The proof of Theorem \ref{thm inv} is complete.
\end{proof}

\begin{Ass}[Second conservation law]\label{Ass:2nd:CL}
The equation \eqref{Eq:Ham:Gen} admits a conservation law of the form
\begin{align}
M(u)=\frac{1}{2}\|u\|_{H^\gamma}^2,
\end{align}
where $\gamma<s.$
\end{Ass}

\begin{Thm}\label{Thm:Set:inv}
Under the Assumption \ref{Ass:2nd:CL} we have the following
\begin{align}
\phi_t\Sigma =\Sigma.
\end{align}
\end{Thm}
The proof of Theorem \ref{Thm:Set:inv} relies on the lemma \ref{Lem:Inv} below. This lemma roughly states that the evolution under $\phi^N_t$ of a set $\Sigma^N_{r,i}$ is a set of the same kind, giving then the essential mechanism needed for the invariance. The remaining part relies on the definition of the set $\Sigma$ and on a limiting argument.  
\begin{Lem}\label{Lem:Inv}
Let $r\leq s$, for any $\gamma<r_1<r$ and any $t$, there exists $i_1:=i_1(t)$ such that for any $i$, if $u_0\in\Sigma^N_{s,i}$ then $\phi^N_tu_0\in\Sigma^N_{r_1,i+i_1}$.
\end{Lem}
\begin{proof}[Proof of Lemma \ref{Lem:Inv}]
We can assume without loss of generality that $t\geq 0$. For $u_0\in \Sigma^N_{r,i}$, we have, by definition, that
\begin{align}
\|\phi^N_{t_1}u_0\|_{H^r}\leq 2\tilde{g}^{-1}(i+j)\quad \forall t_1\leq e^j.
\end{align}
Let us pick a function $i_1:=i_1(t)$ satisfying $e^{j}+t\leq e^{i_1+j}$ for any $j\geq 0.$ Then
\begin{align}
\|\phi^N_{t_1+t}u_0\|_{H^r}\leq 2\tilde{g}^{-1}(i+i_1+j)\quad \forall t+t_1\leq e^{j+i_1}.
\end{align}
But for $t+t_1\leq e^{j+i_1}$, one readily has $t_1\leq e^{j+i_1}-t$. On the other hand we have, from the definition of $i_1$, that $e^{j+i_1}-t\geq e^{j}$. From these two facts, we obtain
 \begin{align}
 \|\phi^N_{t_1+t}u_0\|_{H^r}\leq 2\tilde{g}^{-1}(i+i_1+j)\quad \forall t_1\leq e^j.
 \end{align}
 Since $u_0\in\Sigma^N_{i,s}$, using the estimate \eqref{Est:Bourg:N}, we have
 \begin{align}
 \|\phi^N_{t}u_0\|_{H^r}\leq 2\tilde{g}^{-1}(1+i+\ln(1+|t|)).
 \end{align}
 Hence, using Assumption \ref{Ass:2nd:CL}, we have
 \begin{align}
 \|u_0\|_{H^\gamma}\leq\|u_0\|_{H^r}\leq 2\tilde{g}^{-1}(1+i).
 \end{align}
 Using the conservation of the $H^\gamma-$norm (Assumption \ref{Ass:2nd:CL}), we obtain
 \begin{align}
 \|\phi^N_{t+t_1}u_0\|_{H^\gamma}\leq 2\tilde{g}^{-1}(1+i).
 \end{align}
 For $\gamma<r_1<r$, by interpolation there is $0<\theta<1$ such that
\begin{align}
\|\phi^N_{t+t_1}u_0\|_{H^{r_1}}\leq \|\phi^N_{t+t_1}u_0\|_{H^\gamma}^{1-\theta}\|\phi^N_{t+t_1}u_0\|_{H^{s_1}}^{\theta}\leq 2(\tilde{g}(1+i))^{1-\theta}(\tilde{g}^{-1}(i+i_1+j))^{\theta}.
\end{align}
As $\tilde{g}^{-1}$ is increasing, we take $i_1$ large enough to obtain the bound
\begin{align}
\|\phi^N_{t+t_1}u_0\|_{H^{r_1}}\leq \tilde{g}^{-1}(i+i_1+j) \quad \forall t_1\leq e^j.
\end{align}
Thus $\phi^N_{t+t_1}u_0\in B^N_{r_1,i+i_1,j}$ for all $t_1\leq e^{j}$, for all $j\geq 0.$ We then have $\phi^N_{t}u_0\in\Sigma^N_{r_1,i+i_1}.$ Then the proof of Lemma \ref{Lem:Inv} is finished.
\end{proof}
\begin{Rmk}
In some situation, Assumption \ref{Ass:2nd:CL} is not necessary. If the control on the $H^s$-norm provides a control on the energy, we can proof Lemma \ref{Lem:Inv} without any supplementary conservation law.
\end{Rmk}

Now let us finish the proof of Theorem \ref{Thm:Set:inv}. Since any $\Sigma_r$ is of full $\mu-$measure and the intersection is countable, we obtain the first statement.

Let us take $u_0\in \Sigma$, then $u_0$ belong to each $\Sigma_r$, $r\in l.$ First, consider $u_0\in \Sigma_{i_r}$. There is $i\geq 1$ such that $u_0\in \Sigma_{i_r}$, therefore $u_0$ is the limit of a sequence $(u_{0}^N)$ such that $u_{0}^N\in \Sigma_{i,r}^N$ for every $N$. Now thanks to Lemma \ref{Lem:Inv}, there is $i_1:=i_1(t)$ such that $\phi^N_t(u_{0}^N)\in \Sigma_{i+i_1,r_1}^{N}$.
Using the convergence \eqref{Conv:flows}, we see that $\phi^t(u_0)\in \Sigma_{i+i_1,r_1}$. Now if $u_0\in \partial\Sigma_{i,r}$, there is $(u_0^k)_k\subset\Sigma_{i,r}$ that converges to $u_0$ in $H^r$.  Since we showed that $\phi_t\Sigma_{i,r}\subset\Sigma_{i+i_1,r_1}$ and $\phi_t(\cdot) $ is continuous, we see that $\phi_t(u_0)=\lim_{k}\phi_t(u_0^k)\in \overline{\Sigma_{i+i_1,r_1}}$. One arrives at $\phi_t\overline{\Sigma_{i,r}}\subset\overline{\Sigma_{i+i_1,r_1}}\subset \Sigma_{r_1}$. Hence $\phi_t\Sigma\subset\Sigma.$\\
Now, let $u$ be in $\Sigma$. Since $\phi_t$ is well-defined on $\Sigma$ we can set  $u_0=\phi_{-t}u$, we then have $u=\phi_tu_0$, hence $\Sigma\subset\phi_{t}\Sigma.$ 
\subsection{Construction of the required invariant measures}\label{Subsection: Constr}
Here we present a strategy to achieve the fulfillment of Assumption \ref{Ass3}. It is based on a fluctuation-dissipation method.  

Let us introduce a probabilistic setting and some notations that are going to be used from this section on. First, $(\Omega, \mathcal{F}, \P)$ is a complete probability space. If $E$ is a Banach space, we can define random variables $X:\Omega\to E$ as Bochner measurable functions with respect to $\mathcal{F}$ and $\mathcal{B}(E)$, where $\mathcal{B}(E)$ is the Borel $\sigma-$algebra of $E.$

For every positive integer $N$, we denote by $W^N$ an $N-$dimensional Brownian motion with respect to the filtration $(\mathcal{F}_t)_{t\geq 0}$.

Throughout the sequel, we use $\langle \cdot,\cdot\rangle$ as the real dot product in $L^2$:
\begin{align*}
\langle u,v\rangle =\re\int u\bar{v} \, dx, \quad \text{for all $u,v\in L^2$}.
\end{align*}
 Ocassionally this notation is used to write a duality bracket when the context is not confusing.

We set the fluctuation-dissipation equation
\begin{align}
\dt u=\Pi^NJH'(u)-\al\mathcal{L}(u)+\sqrt{\al}dW^N,\label{Eq:Ham:N:al}
\end{align}
where $\mathcal{L}(u)$ is a dissipation operator.
\begin{Ass}[Exponential moment]\label{Ass:diss:ineq}
We assume that $\mathcal{L}(u)$ satisfies the dissipation inequality
\begin{align}
\mathcal{E}(u)=E'(u,\mathcal{L}(u))\geq e^{\xi^{-1}(\|u\|_{H^s})}\|u\|_{H^{s_0}}^2,
\end{align}
where $\xi^{-1}(x)\geq x$ is a one-to-one convex function from $\R_+$ to $\R_+$.\\
We assume also that the (local existence) function $f$ defined in Assumption \ref{Ass1} is of the form $Cx^{-r}$ for some $r>0$.
\end{Ass}
\begin{Ass}\label{Ass:stGWP}
The equation \eqref{Eq:Ham:N:al} is stochastically globally well-posed on $E^N,$ that is:\\
For every $\mathcal{F}_0$-measurable random variable $u_0$ in $E_N$, we have
\begin{enumerate}
\item
for $\mathbb{P}$-almost all $\omega \in \Omega$, \eqref{Eq:Ham:N:al} with initial condition $u(0) = \Pi^N  u_0^{\omega}$ has a unique global solution $u_{\alpha}^N (t; u_0^{\omega})$;

\item
if $u_{0,n}^{\omega} \to u_0^{\omega}$ then $u_{\alpha}^N (\cdot; u_{0,n}^{\omega}) \to u_{\alpha}^N(\cdot; u_0^{\omega})$ in $C_t E^N$;

\item
the solution $u^N_\al$  is adapted to $(\mathcal{F}_t)$.
\end{enumerate}
\end{Ass}
Under Assumption \eqref{Ass:stGWP}, we can introduce the transition probability 
\begin{align}
T_{t,\al}^N(v,\Gamma)=\P(u^N_\al(t;v)\in\Gamma)\quad\text{where $v\in L^2$ and $\Gamma\in$ Bor($L^2$)}.
\end{align}
We then define the Markov operators
\begin{align}
P^N_{t,\al}f(v) &=\int_{L^2}f(w)T^N_{t,\al}(v,dw)\quad L^\infty(L^2,\R)\to L^\infty(L^2,\R)\\
P^{N*}_{t,\al}\lambda(\Gamma) &=\int_{L^2}\lambda(dw)T^N_{t,\al}(w,\Gamma)\quad \mathfrak{p}(L^2)\to \mathfrak{p}(L^2).
\end{align}
Let us remark that 
\begin{align}
P^N_{t,\al}f(v)=\E f(u^N_\al(t,v)).\label{Rem:SG}
\end{align}
Using the continuity of the solution $u^N_\al(t,\cdot)$ in the second variable (the initial data), we observe the Feller property:
\begin{align}
P^N_{t,\al}C_b(L^2,\R)\subset C_b(L^2,\R).
\end{align}
Here $C_b(L^2,\R)$ is the space of continuous bounded functions$f:L^2\to\R$.

Let us also define the Markov groups associated to \eqref{Eq:Ham:N}, recall that in Assumption \eqref{Ass1} we assume that $\phi^N_t$ exists globally in $t.$ The associated Markov groups are the following
\begin{align}
P^N_tf(v) &=f(\phi^N_t\Pi^Nv)\quad C_b(L^2,\R)\to C_b(L^2,\R)\\
P^{N*}_t\lambda(\Gamma) &=\lambda(\phi^N_{-t}(\Gamma))\quad \mathfrak{p}(L^2)\to \mathfrak{p}(L^2).
\end{align}
Combining the It\^o formula, the Krylov-Bogoliubov argument (Lemma \ref{Lem:KBarg}), the Prokhorov theorem (Theorem \ref{Thm:Prokh}) and Lemma \ref{Lem:MeanMeas} we obtain the following:
\begin{Thm}
Under  Assumptions \ref{Ass:diss:ineq} and \ref{Ass:stGWP}, let $N\geq 1$, there is an stationary measure $\mu^N_\al$ for \eqref{Eq:Ham:N:al} such that
\begin{align}
\int_{L^2}e^{\xi^{-1}(\|u\|_{H^s})}\|u\|_{H^{s_0}}^2\mu^N_\al(du) \leq \int_{L^2}\mathcal{E}(u)\mu^N_\al(du)\leq C,\label{Est:E:N:al}
\end{align}
where $C$ does not depend on $(N,\al)$.\\
There is a subsequence $\{\al_k\}$ such that
\begin{align}
\lim_{k\to\infty} \mu_{\al_k}^N &=\mu^N \quad\text{weakly, for any $N$.}
\end{align}
\end{Thm}

\begin{Ass}[Inviscid limit]\label{Ass:unif:conv}
The following convergence holds
\begin{align}
\lim_{k\to\infty}P^{N*}_{t,\al_k}\mu_{\al_k}^N &=P^{N*}_t\mu^N \quad\text{weakly, for any $N$.}
\end{align}
\end{Ass}

\begin{Thm}\label{thm 3.11}
Under Assumptions \ref{Ass:diss:ineq} and \ref{Ass:unif:conv},
the measures $\mu^N$ are invariant under \eqref{Eq:Ham:N} respectively. They satisfy
\begin{align}
\int_{L^2}e^{\xi^{-1}(\|u\|_{H^s})}\|u\|_{H^{s_0}}^2\mu^N(du)\leq\int_{L^2}\mathcal{E}(u)\mu^N(du)\leq C,\label{Est:exp:N}
\end{align}
where $C$ is independent of $N$.\\
Moreover, we have that
\begin{align}
\lim_{i\to\infty}e^{-2i}\sum_{j\geq 0}\frac{e^{-j}}{f\circ\xi(2(i+j))}=0.\label{Fulfill.limit}
\end{align}
\end{Thm}
\begin{proof}[Proof of Theorem \ref{thm 3.11}]
For simplicity we write the measure $\mu^N_{\al_k}$ constructed above as $\mu^N_k$.
The invariance of $\mu^N$ under \eqref{Eq:Ham:N} follows from the following diagram
\begin{align}
\hspace{10mm}
\xymatrix{
  P_{t,k}^{N*}\mu_k^N \ar@{=}[r]^{(I)} \ar[d]^{(III)} & \mu_k^N \ar[d]^{(II)} \\
    P^{N*}_{t}\mu^N \ar@{=}[r]^{(IV)} & \mu^N
  }
\end{align}
The step $(I)$ is the stationarity of $\mu^N_k$ under \eqref{Eq:Ham:N:al}, the step $(II)$ describes the weak convergence of $(\mu^N_k)$ to $\mu^N$. Therefore, with the help of the weak convergence in $(III)$ (Assumption \ref{Ass:unif:conv}), we obtain the desired invariance in $(IV)$. 

As for the proof of \eqref{Est:exp:N}, we first remark that the weak convergence $\mu^N_k\to\mu^N$ combined with the lower semi-continuity of $\mathcal{E}(u)$ yields
\begin{align}
\int_{L^2}e^{\xi^{-1}(\|u\|_{H^s})}\|u\|_{H^{s_0}}^2\mu^N(du)\leq\int_{L^2}\mathcal{E}(u)\mu^N(du)\leq C,\label{Est:E:N}
\end{align}
with $C$ is the same constant in \eqref{Est:E:N:al}, in particular $C$ does not depend on $N.$\\
Now using the form of $f$, we have that
\begin{align}
f\circ\xi(i+j)=C(\xi(i+j))^{-r}.
\end{align}
Now recall that $\xi^{-1}\geq x$, then $\xi\leq x^{-1}$. So, we obtain $f\circ\xi(2(i+j))\ggeq (i+j)^{-r}$, which leads readily to \eqref{Fulfill.limit}. Then the proof of Theorem \ref{thm 3.11} is finished.
\end{proof}
From now, Assumption \ref{Ass3} is disintegrated into Assumptions \ref{Ass:diss:ineq} to \ref{Ass:unif:conv}.
\section{Proof of Theorem \ref{asGWP:1}: Probabilistic part}\label{Sect:Pr as GWP 1}
In this section, we present the probabilistic part of the proof of Theorem \ref{asGWP:1}, that is, the fulfillment of  Assumptions \ref{Ass:2nd:CL} to \ref{Ass:unif:conv}. First, we remark that Assumption \ref{Ass3} was decomposed in the strategy given in Subsection \ref{Subsection: Constr} into Assumptions \ref{Ass:diss:ineq} to \ref{Ass:unif:conv}. Assumptions \ref{Ass1} and \ref{Ass2} are of deterministic type and will be given in Section \ref{Sect:LWP}. We also prepare, simultaneously, the setting for the proof of Theorem \ref{asGWP:2}. Assumption \ref{Ass:2nd:CL} is given by the existence of a second conservation law (the $L^2$-norm).

Let $N$ be a positive integer, consider the Galerkin approximation
\begin{align}
\dt u =-\i((-\Delta)^\sigma u +\Pi^N \abs{u}^{2} u). \label{fNLS-Glk}
\end{align}
We define a Brownian motion following the case we set the equation on the unit ball $B^d$:
\begin{align}
W^N (t,x) = 
\sum_{n=1}^N a_n e_n (|x|) \beta_n (t),
\label{Noise}
\end{align}
where $(e_n)_n$ is the sequence of eigenfunctions of radial Dirichlet Laplacian $-\Delta$ on $B^d$.
Here $(\beta_n (t))$ is a fixed sequence of independent one dimensional Brownian motions with filtration $(\mathcal{F}_t)_{t\geq 0}$. The numbers $(a_m)_{m\geq 1}$ are complex numbers such that $|a_m|$ decreases sufficiently fast to $0.$ More precisely we assume that
\begin{align}
A_r := \sum_{n \geq 1} z_n^{2r}|a_n|^2 < + \infty\quad\text{for $r\leq \sigma$,}\label{Noise:Cond}
\end{align}
where $(z_n^2)$ are the eigenvalues of $-\Delta$ associated to $(e_n)$. Set 
\begin{align*}
A_r^N = \sum_{n=1}^N z_n^{2r}|a_n|^2.
\end{align*}

Let us introduce the following fluctuation-dissipation model of \eqref{fNLS-Glk}. For $\sigma\in (0,1]$, we define a dissipation operator  as follows
\begin{align}
\mathcal{L}(u)=\begin{cases}(-\Delta)^{s-\sigma}u+e^{\xi^{-1}(\|u\|_{H^{s-}})}u, & \text{if $s>\frac{d}{2}$}\\
e^{\xi^{-1}(\|u\|_{H^{s-}})}\left[(-\Delta)^{s-\sigma}u+\Pi^N|u|^2u\right] &\text{if $0< s\leq 1+\sigma$} 
\end{cases}\label{Lu}
\end{align}
where $\xi:\R_+\to\R_+$ is any one-to-one increasing concave function. We assume that $\xi^{-1}(x)\geq x.$ We then write the stochastic equation
\begin{align}\label{SNLS1}
du = -\i ((-\Delta)^\sigma u +\Pi^N \abs{u}^{2} u) \, dt - \alpha \mathcal{L}(u) \, dt +  \sqrt{\alpha} \, d W^N
\end{align}
where $\alpha \in (0,1)$.
\begin{Rmk}
The setting of Theorem \ref{asGWP:2} is contained in the case $0<s\leq 1+\sigma:$
\begin{align}
\mathcal{L}(u)=e^{\xi^{-1}(\|u\|_{H^{s-}})}\left[(-\Delta)^{s-\sigma}u+\Pi^N|u|^2u\right].
\end{align}
Therefore the estimates obtained using that form of $\mathcal{L}(u)$ will also serve in the proof of Theorem \ref{asGWP:2} (see Section \ref{Sect:Pr as GWP2}).
\end{Rmk}
\subsection{Fulfillment of Assumption \ref{Ass:diss:ineq} (The dissipation inequality and time of local existence)}
For the polynomial nature of the time of local existence we refer to Theorem \ref{thm LWP}.

For the dissipation inequality, we set $\mathcal{E}(u):=E'(u,\mathcal{L}(u))$. We have the following.
 
\begin{enumerate}
\item For $s>\frac{d}{2}$,
\begin{align}
\mathcal{E}(u) &= \|u\|_{H^s}^2+\langle (-\Delta)^{s-\sigma}u,\Pi^N|u|^2u\rangle +e^{\xi^{-1}(\|u\|_{H^{s -}})}\left(\|u\|_{L^4}^4 +\|u\|_{H^\sigma}^2\right).
\end{align}
Using Cauchy-Schwarz and Agmon's inequalities and the algebra property of $H^{s-}$ for $s>\frac{d}{2}$, we obtain
\begin{align}
|\langle(-\Delta)^{s-\sigma}u,\Pi^N|u|^{2}u\rangle| \leq\|u\|_{H^{s-\sigma}}\||u|^2u\|_{H^{s-\sigma}}\leq C\|u\|_{L^\infty}\||u|^2u\|_{H^{s-\sigma}} &\leq C\|u\|_{L^2}^{1-\frac{d}{2s}}\|u\|_{H^s}^{\frac{d}{2s}}\||u|^2u\|_{H^{s-\sigma}}\leq C\|u\|_{H^s}\|u\|_{H^{s-}}^3 \\
&\leq \tilde{C}+\frac{1}{2}\|u\|_{H^s}^2\|u\|_{H^{s-}}^{6}
\leq \tilde{C}+\frac{1}{2}e^{\xi^{-1}(\|u\|_{H^{s -}})}\|u\|_{H^s}^2.
\end{align}
Remark that the $\tilde{C}$ is an absolute constant. Overall,
\begin{align}
\mathcal{E}(u)\geq \|u\|_{H^{s}}^2+\frac{1}{2}e^{\xi^{-1}(\|u\|_{H^{s -}})}\left(\|u\|_{L^4}^4 +\|u\|_{H^s}^2\right)-\tilde{C}.\label{Est_low_tildeC}
\end{align}
\item For $0< s\leq 1+\sigma$,
\begin{align}
\mathcal{E}(u)=e^{\xi^{-1}(\|u\|_{H^{s -}})}\left(\|u\|_{H^s}^2 +\langle(-\Delta)^{s-\sigma}u,\Pi^N|u|^{2}u\rangle +\langle\Pi^N|u|^2u,(-\Delta)^\sigma u\rangle +\|\Pi^N|u|^2u\|_{L^2}^2\right).
\end{align}
We are able to give a useful estimate of $\mathcal{E}(u)$ in the following two cases:
\begin{enumerate}
\item For $\max(\frac{1}{2},\sigma)\leq s\leq 1+\sigma:$

In order the treat the term $\langle(-\Delta)^{s-\sigma}u,\Pi^N|u|^{2}u\rangle$ and $\langle\Pi^N|u|^2u,(-\Delta)^\sigma u\rangle$, we use the C\'ordoba-C\'ordoba inequality ( see Corollary \ref{Lem:Cplx:Cor}) since $s-\sigma\in [0,1]$, $s$ being in $[\max(\frac{1}{2},\sigma),1+\sigma]$. We have the following
\begin{align}
\langle(-\Delta)^{s-\sigma}u,|u|^{2}u\rangle &\geq \frac{1}{2}\|(-\Delta)^\frac{s-\sigma}{2}|u|^2\|_{L^2}^2,\\
\langle(-\Delta)^{\sigma}u,|u|^{2}u\rangle &\geq \frac{1}{2}\|(-\Delta)^\frac{\sigma}{2}|u|^2\|_{L^2}^2.
\end{align}
 We finally obtain that
\begin{align}
 \mathcal{E}(u) &\geq \frac{1}{2}e^{\xi^{-1}(\|u\|_{H^{s -}})}\left( \||u|^2\|_{\dot{H}^\sigma}^2+\|\Pi^N|u|^2u\|_{L^2}^2+\|u\|_{H^s}^2+\||u|^2\|_{\dot{H}^{s-\sigma}}^2\right).  \label{Est:Ener:low} 
\end{align}

\item For $0<s\leq \sigma:$ Since $\sigma$ still lies in $(0,1]$, we can use the C\'ordoba-C\'ordoba inequality to obtain
\begin{align}
\langle(-\Delta)^{\sigma}u,|u|^{2}u\rangle &\geq \frac{1}{2}\|(-\Delta)^\frac{\sigma}{2}|u|^2\|_{L^2}^2.
\end{align}
However, we cannot use the C\'ordoba-C\'ordoba inequality for $(-\Delta)^{s-\sigma}$ in the term $\langle(-\Delta)^{s-\sigma}u,|u|^{2}u\rangle $ since $s\leq \sigma$.  We instead employ a direct estimation (using in particular the Sobolev embedding $L^4=W^{0,4}\subset W^{s-\sigma,4}$):
\begin{align}
\langle(-\Delta)^{s-\sigma}u,|u|^{2}u\rangle &\geq-\|(-\Delta)^{s-\sigma}u\|_{L^4}\|u\|_{L^4}^3\geq -C\|u\|_{L^4}^4.
\end{align}
We arrive at
\begin{align}
 \mathcal{E}(u) &\geq \frac{1}{2}e^{\xi^{-1}(\|u\|_{H^{s -}})}\left( \||u|^2\|_{\dot{H}^\sigma}^2+\|\Pi^N|u|^2u\|_{L^2}^2+\|u\|_{H^s}^2-C\|u\|_{L^4}^4\right).  
\end{align}
Now,  let us set
\begin{align}
\mathcal{M}(u):=M'(u,\mathcal{L}(u)),
\end{align}
where $M(u)=\frac{1}{2}\|u\|_{L^2}^2.$
We obtain that
\begin{align}
\mathcal{M}(u)=e^{\xi^{-1}(\|u\|_{H^{s-}})}\left(\|u\|_{L^4}^4+\|u\|_{H^{s-\sigma}}^2\right).
\end{align}
This gives us that
\begin{align}
 \mathcal{E}(u) &\geq \frac{1}{2}e^{\xi^{-1}(\|u\|_{H^{s -}})}\left( \||u|^2\|_{\dot{H}^\sigma}^2+\|\Pi^N|u|^2u\|_{L^2}^2+\|u\|_{H^s}^2\right)-\frac{C}{2}\mathcal{M}(u).  \label{Est:Ener:lowbis} 
\end{align}
\end{enumerate}
\end{enumerate}

\begin{Rmk}
Notice that for $\sigma<\frac{1}{2}$, we do not have an estimate of $\mathcal{E}(u)$ for $s\in (\sigma,d/2)$.  So we construct global solutions in that case using the arguments in Section \ref{Sect:Pr as GWP2}.  Surprisingly enough,  we construct solutions, for such $\sigma$,  in lower regularities $s\in (0,\sigma)$ since in that cases the need estimates are obtained through the control of $\mathcal{E}(u)$ established above (see Section \ref{Sect:Pr as GWP2}).
\end{Rmk}

For any $\sigma\in (0,1]$, Assumption \ref{Ass:diss:ineq} is now fulfilled for $s\in (d/2,\infty)\cup[\max(1/2,\sigma),1+\sigma]\cup[0,\sigma]$. Notice that the term $-\frac{C}{2}\mathcal{M}(u)$ in \eqref{Est:Ener:lowbis} does not have any obstruction because it is controlled under expectation (see \eqref{Est:M:al:N} for details of the It\^o estimate on $M(u)$). It's expectation is just added the to right hand side of \eqref{Est:E:N:al} and similarly for absolute $\tilde{C}$ in \eqref{Est_low_tildeC}.

In particular, we notice that
\begin{align}\label{Est_bound:on:E}
 \|u^{2}\|_{\dot{H}^\sigma}^2+\|u\|_{H^s}^2+\|\Pi^N|u|^2u\|_{L^2}^2\leq (\|u^{2}\|_{\dot{H}^\sigma}^2+\|u\|_{H^s}^2+\|\Pi^N|u|^2u\|_{L^2}^2)e^{\xi^{-1}(\|u\|_{H^{s -}})}\leq 2\mathcal{E}(u)+C\mathcal{M}(u),
\end{align}
where $C=C(s,\sigma)$ does not depend on $(u,N,\al)$.

\subsection{Fulfillment of Assumption \ref{Ass:stGWP} (The stochastic GWP)} For simplicity we present the proof of stochastic GWP in the case $0\leq s\leq 1+\sigma$, we then consider the corresponding dissipation $\mathcal{L}(u)$. The other case involves a simpler dissipation operator and can be treated similarily.

We follow the well-known Da Prato-Debussche decomposition\cite{da2002two} (also known as Bourgain decomposition \cite{bourgNLS96}). Let us consider the following linear stochastic equation
\begin{align}
dz=-\i(-\Delta)^\sigma zdt+\sqrt{\al}dW^N, \quad z\Big|_{t=0}=0.
\end{align}
This equation admits a unique solution known as stochastic convolution:
\begin{align}
z(t)=\sqrt{\al}\int_0^te^{-\i(t-s)(-\Delta)^\sigma}dW^N(s).\label{St:Conv}
\end{align}
It follows the It\^o theory that $z(t)$ is an $\mathcal{F}_t-$martingale. The Doob maximal inequality combined with the It\^o isometry shows that, for any $T>0$
\begin{align}
\E\sup_{t\in [0,T]}\|z(t)\|_{L^2}\leq C_T\sqrt{\al}.\label{Est:on:z}
\end{align} 
We notice that $z$ exists globally in $t$ for almost all $\omega\in\Omega$. Let $u_0$ be $E^N-$valued $\mathcal{F}_0-$independent random variable. Fix $\omega$ such $z^\omega$ exists globally in time, and consider the nonlinear deterministic equation
\begin{align}
\dt v = F(v,z),\quad v(t=0)=u_0^\omega,\label{Eq:nonlin}
\end{align}
where $F(v,z)=-\i((-\Delta)^\sigma v+|v+z|^2(v+z))-\al e^{\xi^{-1}(\|v+z\|_{H^{s-}})}[\Pi^N|v+z|^2(v+z)+(-\Delta)^{s-\sigma}(v+z)]$.

We remark that for a global solution $v$ to \eqref{Eq:nonlin}, we have that $u=v+z$ is a global solution to \eqref{SNLS1} supplemented with the initial condition $u(t=0)=u_0^\omega$.

The local existence for \eqref{Eq:nonlin} follows from the classical Cauchy-Lipschitz theorem since $F(v,z)$ is smooth. Let us show that this local solution is in fact global. By an standard iteration argument, it suffices to show that the $L^2$-norm does not blow up in finite time. We have that
\begin{align}
\dt \|v\|_{L^2}^2=-2\i\langle v,\Pi^N|v+z|^2(v+z)\rangle-2\al e^{\xi^{-1}(\|v+z\|_{H^{s-}})}[\langle v,\Pi^N|v+z|^2(v+z)\rangle+ \langle v,(-\Delta)^{s-\sigma}(v+z)\rangle].
\end{align}
By adding $z-z$ and using Cauchy-Schwarz, we obtain
\begin{align}
\dt \|v\|_{L^2}^2&\leq \frac{C_N\|z\|_{L^2}^4}{\al}+e^{\xi^{-1}(\|v+z\|_{H^{s-}})}\left(-\|v\|_{H^{s-\sigma}}^2+\tilde{C}_{N}(\|z\|_{L^2}^2+\|z\|_{L^4}^4)\right).
\end{align}
We then have two complementary cases for a $t\in [0,T]$:
\begin{enumerate}
\item either $\|v\|_{H^{s-\sigma}}^2\leq \tilde{C}_N(\|z\|_{L^2}^2+\|z\|_{L^4}^4)$; in this case 
\begin{align}
\dt \|v\|_{L^2}^2 \leq C^0_\al(\omega,T), 
\end{align}
\item or $\|v\|_{H^{s-\sigma}}^2> \tilde{C}_N(\|z\|_{L^2}^2+\|z\|_{L^4}^4)$; in this case 
\begin{align}
e^{\xi^{-1}(\|v+z\|_{H^{s-}})}\left(-\|v\|_{H^{s-\sigma}}^2+\tilde{C}_N\|z\|_{L^2}^2\right)\leq 0,
\end{align}
and
\begin{align}
\dt \|v\|_{L^2}^2 \leq C^1_\al(\omega,T).
\end{align}
Taking $C_\al=\max(C^0_\al,C^1_\al)$, we obtain the wished finiteness of $\|v\|_{L^2}$.
\end{enumerate}
Now, using the mean value theorem and the Gronwall lemma, we find for $u$ and $v$ two solutions to \eqref{SNLS1} that
\begin{align}
\|u(t)-v(t)\|_{L^2}\leq \|u(0)-v(0)\|_{L^2}e^{\sup_{x\in D}\int_0^1|f'(ru+(1-r)v)|dr},
\end{align} 
where $f(u)=F(u,0).$\\
Standard arguments show that the constructed solution $u$ is adapted to $\sigma(u_0,\mathcal{F}_t).$
\subsection{Fulfillment of Assumption \ref{Ass:unif:conv} (Zero viscosity limit, $\alpha\to 0$)}
The difficulty in the convergence of Assumption \ref{Ass:unif:conv} is the fact that both $P^{N*}_{t,\al_k}$ and $\mu^N_{\al_k}$ depend on $\al_k$, we then need some uniformity in the convergence $u^N_{\al_k}(t,\cdot)\to \phi^N_t(\cdot)$ as $\al_k\to 0$. The following lemma gives the needed uniform convergence. To simplify the notation, we use the abuse of notation: $P^{N*}_{t,\al_k}=:P^{N*}_{t,k}$, $\mu^N_{\al_k}=:\mu^N_{k}$ and $u_{\al_k}=:u_k$.
\begin{Lem}\label{Lem:lim:al}
Let $T>0$. For any $R>0$, any $r>0$,
\begin{align}
\sup_{u_0\in B_R(L^2)}\sup_{t\in [0,T]}\E\left(\|u_k(t,\Pi^N u_0)-\phi^N_t \Pi^Nu_0\|_{L^2}1_{S_{r,k}(t)}\right)\to 0,\quad \text{as $k\to\infty$},
\end{align}
where 
\begin{align}
S_{r,k}(t)=\left\{\omega\in\Omega\ |\ \text{max}\left(\|z_{\al_k}^\omega(t)\|_{L^2},\sqrt{\al_k}\left|\sum_{m=1}^N\int_0^ta_m\langle u,e_m\rangle d\beta_s^\omega\right|\right)\leq r\right\}.
\end{align}
\end{Lem}
We postpone the proof of Lemma \ref{Lem:lim:al}. We remark that, using the Ito isometry and the Chebyshev inequality we have
\begin{align}
\E(1-1_{S_{r,k}}(t))\leq \frac{Ct}{r^2}\label{Est:on:S_r,k}
\end{align}
where $C$ does not depend on $(r,k,t)$.

Let $f:L^2\to\R$ be a bounded Lipschitz function. Without any loss of generality we can assume that $f$ is bounded by $1$ and its Lipschitz constant is also $1$. We have
\begin{align}
\langle P^{N*}_{t,k}\mu^N_k,f\rangle - \langle P^{N*}_{t}\mu^N,f \rangle &= \langle\mu^N_k,P^{N}_{t,k}f\rangle - \langle\mu^N,P^N_tf\rangle\\
& =\langle\mu^N_k,(P^N_{t,k}-P^N_t)f\rangle -\langle\mu^N-\mu^N_k,P^N_tf\rangle\\
&=A-B.
\end{align}
We see that $B\to 0$ as $k\to\infty$ according to the weak convergence $\mu^N_k\to\mu^N.$ \\
Using the boundedness of $f$ we obtain
\begin{align}
|A|\leq \int_{B_R(L^2)}|P^N_{t,k}f(w)-P^N_tf(w)|\mu^N_k(dw)+2\mu^N_k(L^2\backslash B_R(L^2))=:A_1+A_2.
\end{align}
Now from \eqref{Est:E:N:al},
\begin{align}
\int_{L^2}\|u\|_{L^2}^2\mu^N_k(du)\leq \int_{L^2}\mathcal{E}(u)\mu^N_k(du)\leq  C,
\end{align}
where $C$ does not depend on $(N,\al).$ Combining this with the Chebyshev inequality we obtain that
\begin{align}
A_2\leq \frac{2C}{R^2}.
\end{align}
Using the boundedness and Lipschitz properties of $f$, we obtain
\begin{align}
A_1\leq  A_{1,1}+A_{1,2}
\end{align}
where
\begin{align}
A_{1,1} &=\int_{B_R(L^2)}\|\E u_k^N(t,\Pi^Nw)-\phi_tw\|_{L^2}1_{S_r}\mu^N_k(dw)\\
A_{1,2} &= 2\int_{L^2\backslash B_R(L^2)}\E(1-1_{S_r})\mu^N_k(dw).
\end{align}
Using the \eqref{Est:on:S_r,k} we obtain that
\begin{align}
A_{1,2}\leq \frac{2Ct}{r^2}.
\end{align}
We finally obtain
\begin{align}
|A|\leq A_{1,1}+ \frac{2Ct}{r^2}+\frac{2C}{R^2}.
\end{align}
We pass to the limit on $k$ first, by applying apply Lemma \ref{Lem:lim:al}, we obtain $A_{1,1}\to 0$. Then we take the limits $r\to\infty$ and $R\to\infty$, we arrive at the conclusion of Assumption \ref{Ass:unif:conv}. 

Now let us prove Lemma \eqref{Lem:lim:al}:
\begin{proof}[Proof of Lemma \ref{Lem:lim:al}]
Let $u_0\in B_R(L^2)$ and $u_k$ and  $u$ be the solutions of \eqref{SNLS1} (with viscosity $\alpha_k$) and \eqref{fNLS-Glk} starting at $u_0$, respectively. Recall that $u_k$ can be decomposed as $u_k=v_k+z_k$ where $v_k$ is the solution of \eqref{Eq:nonlin} with initial datum $u_0$ and $z_k$ is given by \eqref{St:Conv} with $\alpha_k$. In order to prove Lemma \ref{Lem:lim:al}, it suffice to show that
\begin{align}
\sup_{u_0\in B_R}\sup_{t\in [0,T]}\E\left(\|v_k-u\|_{L^2}1_{S_{r,k}(t)}\right)\to 0, \quad\text{as $k\to\infty$}.
\end{align} 
Indeed, we already have by \eqref{Est:on:z} that $\E\sup_{t\in [0,T]}\|z_k(t)\|_{L^2}\to 0$ as $k\to\infty$. Set $w_k=v_k-u$. We will treat the case $s\leq 1+\sigma$ which is more delicate. We then consider the equation satisfied by $w_k$:
\begin{align}
\dt w_k=\i\Delta w_k +\Pi^N[w_kf(v_k,u,z_k)+z_kg(v_k,z_k)]-\al_k e^{\xi^{-1}(\|v_k+z_k\|_{H^{s-}})}(-\Delta)^{s-\sigma}(v_k+z_k),
\end{align}
where $f$ and $g$ is a homogeneous of degree $2$ complex polynomial. We claim that $\lim_{k\to\infty}\|w_k\|_{L^2}=0$ almost surely. Indeed, by taking the dot product with $w_k$, we obtain after the use of the Gronwall inequality,
\begin{align}
\sup_{t\in[0,T]}\|w_k\|_{L^2}^2\leq e^{T(1+z_N^2+\|f(v_k,z_k)\|_{L^\infty})}\left(\int_0^T\|z\|_{L^2}^2dt+\alpha_kT\right)\left(1+e^{C(N)(\xi^{-1}(\|v_k\|_{L^\infty_tL^2_x})+\xi^{-1}(\|z_k\|_{L^\infty_tL^2_x}))}\right).\label{ConvGENERALEst}
\end{align}
We pass to the limit $k\to\infty$ with the use of \eqref{Est:on:z} to obtain the claim.

Now, writing the It\^o formula for $\|u_k\|_{L^2}^2$, we have
\begin{align*}
\|u_k\|_{L^2}^2+2\al_k\int_0^t\mathcal{M}(u_k) \, d\tau =\|\Pi^Nu_0\|_{L^2}^2+\al_k \frac{A_0^N}{2}t+2\sqrt{\al_k}\sum_{m\leq N}a_m\int_0^t\langle u_k,e_m\rangle \, d\beta_m,
\end{align*}
where
\begin{align*}
\mathcal{M}(u)=M'(u,\mathcal{L}(u))\geq 0.
\end{align*}
Since $\al_k\leq 1,$ we have that, on the set $S_{r,k}(t),$
\begin{align*}
\|u_k\|_{L^2}^2\leq \|\Pi^Nu_0\|_{L^2}^2+C(r,N)t,
\end{align*}
where $C(r,N)$ does not depend on $k.$ Hence we see that, on $S_{r,k}(t)$,
\begin{align*}
\|w_k\|_{L^2}\leq \|v_k\|_{L^2}+\|z_k\|_{L^2}\leq \|u_k\|_{L^2}+2\|z_k\|_{L^2}\leq \|u_0\|_{L^2}+3C(r,N)t.
\end{align*}
In particular, we have the following two estimates:
\begin{align}
\sup_{k\geq 1}\sup_{u_0\in B_R}\|w_k\|_{L^\infty_tL^2_x}\Bbb 1_{S_{r,k}}\leq R+3C(r,N)T,\label{ControlTronqueeStoch}
\end{align}
Hence coming back to \eqref{ConvGENERALEst} and using the (deterministic) conservation $\|u(t)\|_{L^2}=\|P_N u_0\|_{L^2}$ and the estimate \eqref{ControlTronqueeStoch}, we obtain
\begin{align*}
\sup_{u_0\in B_R}\|w_k\|_{L^\infty_tL^2_x}^2\Bbb 1_{S_r}\leq A(R,N,r,T)(\alpha_k+\|z_k\|_{L^1_tL^2_x}).
\end{align*}
Therefore, using again the bound \eqref{Est:on:z}, we obtain the almost sure convergence $\|z_k\|_{L^2}\to 0$ (as $k\to\infty$, up to a subsequence), we obtain then the almost sure convergence
\begin{align*}
\lim_{k\to\infty}\sup_{u_0\in B_R}\|w_k\|_{L^\infty_tL^2_x}^2\Bbb 1_{S_r}=0.
\end{align*}
Now, we use \eqref{ControlTronqueeStoch} and the Lebesgue dominated convergence theorem to obtain
\begin{align*}
\E\sup_{u_0\in B_R}\|w_k\|_{L^\infty_tL^2_x}\Bbb 1_{S_r}\to 0,\quad as\ k\to\infty.
\end{align*}
The proof  of Lemma \ref{Lem:lim:al} is finished.
\end{proof}

\subsection{Estimates}\label{Estimates}
Let us recall that
\begin{align}
\mathcal{M}(u):=M'(u,\mathcal{L}(u))=\langle u,\mathcal{L}(u)\rangle.
\end{align}
\begin{Prop}\label{prop 4.4}
We have that
\begin{align}
\int_{L^2}\mathcal{M}(u)\mu(du)=\frac{A_0}{2}.\label{Est:M}
\end{align}
\end{Prop}
\begin{proof}[Proof of Proposition \ref{prop 4.4}]
We present the more difficult case in which  the dissipation is given by 
\begin{align}
\mathcal{L}(u)=e^{\xi^{-1}(\|u\|_{H^{s-}})}\left(\Pi^N|u|^2u+(-\Delta)^{s-\sigma}u\right), \quad\text{where $0\leq s\leq 1+\sigma$}.
\end{align}
In this case we have 
\begin{align}
\mathcal{M}(u)=e^{\xi^{-1}(\|u\|_{H^{s-}})}\left(\|u\|_{L^4}^4+\|u\|_{H^{s-\sigma}}^2\right).
\end{align}
The other case can be proved using a similar procedure.
  We split the proof in different steps:
\begin{enumerate}
\item {\bfseries Step 1: Identity for the $(\mu^N_\al$).}\\
Applying the It\^o formula to $M(u)$, where $u$ is the solution to \eqref{SNLS1}, we obtain
\begin{align}
dM(u)=\left[-\alpha \mathcal{M}(u)dt+\frac{\al}{2}A_0^N\right]dt+\sqrt{\al}\sum_{m=0}^Na_m\langle u,e_m\rangle d\beta_m.
\end{align}
Integrating in $t$ and in ensemble with repect to $\mu^N_\al$, we obtain
\begin{align}
\int_{L^2}\mathcal{M}(u)\mu^N_\al(du)=\frac{A_0^N}{2},\label{Est:M:al:N}
\end{align}
this identity used the invariance $\mu^N_\al.$\\
Let us now establish an auxiliary bound. Apply the It\^o formula to $M^2(u)$
\begin{align}
dM^2(u)=\left[\al M(u)\left(-\mathcal{M}(u)+\frac{A_0^N}{2}\right)+\frac{\al}{2}\sum_{m\leq N}\langle u,e_m\rangle^2\right]dt+2\sqrt{\al}M(u)\sum_{m\leq N}\langle u,e_m\rangle d\beta_m.
\end{align}
Integrating in $t$ and with respect respect to $\mu^N_\al$ we arrive at
\begin{align}
\int_{L^2}M(u)\mathcal{M}(u)\mu^N_\al(du) \leq C,\label{Est:MM:N:al}
\end{align}
where $C$ does not depend in $(N,\al)$. The estimate above is obtained after the following remark
\begin{align}
\int_{L^2}M(u)\mu^N_\al(du)\leq \int_{L^2}\mathcal{E}(u)\mu^N_\al(du)\leq C_1
\end{align} 
where $C_1$ is independent of $(N,\al)$.
\item {\bfseries Step 2: Identity for the $(\mu^N$).}
By usual arguments, we obtain the following estimates for $(\mu^N)$:
\begin{align}
\int_{L^2}\mathcal{M}(u)\mu^N(du) &=\frac{A_0^N}{2};\label{Est:M:N}\\
\int_{L^2}M(u)\mathcal{M}(u)\mu^N(du) &\leq C.\label{Est:MM:N}
\end{align}
We do not give details of proof of the identities. In the next step we will prove more delicate estimates whose proof is highly more difficult as the passage to the limit cannot use any finite-dimensional advantage.
 
\item {\bfseries Step 3: Identity for the $\mu$.}
In this part of the proof we perform the passage to the limit $N\to\infty$ in \eqref{Est:M:N} in order to obtain the identity \eqref{Est:M}.\\
The inequality 
\begin{align}
\int_{L^2}\mathcal{M}(u)\mu(du)\leq \frac{A_0}{2}
\end{align}
can be obtained by invoking lower semi-continuity of $\mathcal{M}$. The other way around, the analysis is more challenging. In \cite{syyu} a similar identity was established with the use of an auxiliary estimate on a quantity of type $E(u)\mathcal{E}(u)$. But in our context, such an estimate is not available. Indeed, in order to obtain it, we shall apply the It\^o formula on $E^2$, the dissipation will be $E(u)\mathcal{E}(u)$; but other terms including $A_\sigma^NE(u)$ have to be controlled in expectation. However, we do not have an $N-$independent control on $\int_{L^2}E(u)\mu^N(du)$ because this term is smoother than $\mathcal{E}(u)$ for $s<\sigma$. Therefore the latter cannot be exploited. We remark, on the other hand, that this is why the measure $\mu$ is expected to be concentrated on regularities lower than the energy space $H^\sigma$. Now, without a control on $E(u)\mathcal{E}(u)$, we can only handle the weaker quantity $M(u)\mathcal{M}(u)$ from \eqref{Est:MM:N:al}. So our proof here will be more tricky than that in \cite{syyu}.\\ 
Coming back to the proof of the remaining inequality, we write for some fixed frequency $F$, the frequency decomposition (setting $\Pi^{>F}=1-\Pi^F$)
\begin{align}
\frac{A_0^N}{2} &\leq \E e^{\xi^{-1}(\|u^N\|_{H^{s-}})}\left(\|\Pi^F|u^N|^2\|_{L^2}^2+\|\Pi^{F}u^N\|_{H^{s-\sigma}}^2\right)+\E e^{\xi^{-1}(\|u^N\|_{H^{s-}})}\left(\|\Pi^{> F}|u^N|^2\|_{L^2}^2+\|\Pi^{> F}u^N\|_{H^{s-\sigma}}^2\right)
\\
&=:i+ii,
\end{align}
where $u^N$ is distributed by $\mu^N$. Let us follow the following sub-steps:
\begin{enumerate}
\item First, $ii$ can be estimated by using the control on $\mathcal{E}$ \eqref{Est:Ener:low} as follows
\begin{align}
ii &\lleq F^{-\sigma}\left(\E e^{\xi^{-1}(\|u^N\|_{H^{s-}})}(\|\Pi^{> F}|u^N|^2\|_{H^\sigma}^2+ \|u^N\|_{\dot{H}^s}^2)\right)
\lleq F^{-\sigma}\E\mathcal{E}(u^N)\lleq F^{-\sigma}C.
\end{align}

\item As for $i$, we can split it by using localization in $H^{s-}$ (notice that by using the Skorokhod theorem $(u^N)$ converges almost surely on $H^{s-}$ to some $H^{s-}-$valued random variable $u$). For any fixed $R>0$ we set $\chi_R=1_{\{\|u^N\|_{H^{s-}}\geq R\}}$. We have
\begin{align}
&\E e^{\xi^{-1}(\|u^N\|_{H^{s-}})}(\|\Pi^{F}|u^N|^2\|_{L^2}^2 + \|\Pi^Fu^N\|_{H^{s-\sigma}}^2)\chi_R\\
&\lleq \E e^{\xi^{-1}(\|u^N\|_{H^{s-}})}(\|\Pi^{F}|u^N|^2\|_{H^{\frac{d}{2}+}}^2 + \|\Pi^Fu^N\|_{H^{s-\sigma}}^2)\chi_R\\
&\lleq  \E e^{\xi^{-1}(\|u^N\|_{H^{s-}})}(\|u^N\|_{H^{\frac{d}{2}+}}^4 + \|\Pi^Fu^N\|_{H^{s-\sigma}}^2)\chi_R\\
 &\lleq F^{\frac{d}{2}+}\E e^{\xi^{-1}(\|u^N\|_{H^{s-}})}(\|u^N\|_{L^2}^4+ \|u^N\|_{H^{s-\sigma}}^2)\chi_R\nonumber\\
&\lleq F^{\max(\frac{d}{2}+, s-\sigma)}R^{-2}\E e^{\xi^{-1}(\|u^N\|_{H^{s-}})}\|u^N\|_{H^{s-}}^2\left(\|u^N\|_{L^2}^4+ \|u^N\|_{H^{s-\sigma}}^2\right)\nonumber\\
&\lleq F^{\max(\frac{d}{2}+, s-\sigma)+s}R^{-2}\E \|u^N\|_{L^2}^2e^{\xi^{-1}(\|u^N\|_{H^{s-}})}\left[\|u^N\|_{L^4}^4+ \|u^N\|_{H^{s-\sigma}}^2\right].
\end{align}
Let us  use the estimate  \eqref{Est:M:N} and \eqref{Est:MM:N}  to arrive at
\begin{align}
\E e^{\xi^{-1}(\|u^N\|_{H^{s-}})}\left(\|\Pi^F|u^N|^2\|_{L^2}^2+\|\Pi^Fu^N\|_{H^{s-\sigma}}^2\right)\chi_R &\lleq F^{\max(\frac{d}{2}+, s-\sigma)+s}R^{-2}\E M(u^N)\mathcal{M}(u^N)\\
&\lleq F^{\max(\frac{d}{2}+, s-\sigma)+s}R^{-2}C,
\end{align}
where $C$ does not depend on $N$.
 \item On the other hand, we have the following convergence, thanks to the Lebesgue dominated convergence theorem,
\begin{align}
\lim_{N\to\infty} \E e^{\xi^{-1}(\|u^N\|_{H^{s-}})}\left(\|\Pi^F|u^N|^2\|_{L^2}^2+ \|\Pi^Fu^N\|_{H^{s-\sigma}}^2\right)(1-\chi_R)=\E e^{\xi^{-1}(\|u\|_{H^{s-}})}\left(\|\Pi^F|u|^2\|_{L^2}^2+ \|\Pi^Fu\|_{H^{s-\sigma}}^2\right)(1-\chi_R).
\end{align}
\end{enumerate}
Gathering all this, we obtain, after the limit $N\to\infty$, that
\begin{align}
\frac{A_0}{2}\leq \E e^{\xi^{-1}(\|u\|_{H^{s-}})}\left(\|\Pi^F|u|^2\|_{L^2}^2+ \|\Pi^Fu\|_{H^{s-\sigma}}^2\right)(1-\chi_R)+ F^{-\sigma} C_1+F^{\max(\frac{d}{2}+, s-\sigma)+s}R^{-2}C_2.
\end{align}
We let $R\to\infty$, then $F\to\infty$ and obtain
\begin{align}
\frac{A_0}{2}\leq e^{\xi^{-1}(\|u\|_{H^{s-}})}\E\left(\|u\|_{L^4}^4+ \|u\|_{H^{s-\sigma}}^2\right).
\end{align}
\end{enumerate}
The proof of Proposition \ref{prop 4.4} is finished.
\end{proof}
\begin{Rmk}\label{Rmk:LargeD}
The identity \eqref{Est:M} is crucial for establishing the non-degeneracy properties of the measure $\mu.$ It trivially rule out the Dirac measure at $0$, notice that the Dirac at $0$ is a trivial invariant measure for FNLS.\\
Also, by considering a noise $\kappa dW$, \eqref{Est:M} becomes
\begin{align}
\E\mathcal{M}(u)=\frac{\kappa A_0}{2}.\label{Kappa}
\end{align}
Such scaled noises provide invariant measures $\mu_\kappa$ for FNLS, all satisfying \eqref{Kappa}. Define a cumulative measure $\mu^*$ by a convex combination
\begin{align}
\mu^*=\sum_{j=0}^\infty\rho_j\mu_{\kappa_j}
\end{align}
where $\kappa_j\uparrow\infty$ and $\sum_0^\infty\rho_j=1$. The measure $\mu^*$ is invariant for FNLS. Moreover, for any $K>0$, we can find a positive $\mu^*$-measure set of data whose $H^s$-norms are larger than $K.$  
\end{Rmk}

In order to finish the proof of Theorem \ref{asGWP:1}, we present in the section below the fulfillment of Assumptions \ref{Ass1} and \ref{Ass2}.

\section{End of the proof of Theorem \ref{asGWP:1}: Local well-posedness}\label{Sect:LWP}
In this section, we present a deterministic local well-posedness result for $\sigma \in [\frac{1}{2} ,1]$ in \eqref{fNLS}, which heavily replies on a bilinear Strichartz estimate obtained in Subsection \ref{ssect bilinear}. See also Section 3 in \cite{syyu2}. We also show a convergence from Galerkin approximations of FNLS to FNLS.

\begin{Thm}[Deterministic local well-posedness on the unit ball]\label{thm LWP}
The fractional NLS \eqref{fNLS} with $\sigma \in [\frac{1}{2} ,1]$ is  locally well-posed for radial data $u_0 \in H_{rad}^s (B^d)$, $s> s_l (\sigma)$, where $s_l(\sigma)$ is defined as in \eqref{eq s*}. More precisely, let us first fix $s > s_l (\sigma)$ (defined as in \eqref{eq s*}), and for every $R> 0$, we set $\delta = \delta(R) = c R^{-2s}$ for some $c \in (0,1]$. Then there exists $b> \frac{1}{2}$ and $C , \widetilde{C} >0$ such that every $u_0 \in H_{rad}^s (B^d)$ satisfying $\norm{u_0}_{H_{rad}^s (B^d)} \leq R$, there exists a unique solution of \eqref{fNLS} in $X_{\sigma, rad}^{s,b} ([-\delta, \delta] \times B^d)$ with initial condition $u(0) = u_0$. 
Moreover,
\begin{align*}
\norm{u}_{L_t^{\infty} H_x^{s} ([-\delta, \delta] \times B^d) } \leq C \norm{u}_{X_{\sigma, rad}^{s,b}([-\delta,\delta] \times B^d)} \leq  \widetilde{C} \norm{u_0}_{H_{rad}^s (B^d)} . 
\end{align*}
\end{Thm}
\begin{Rmk}
The function $f$ in Assumption \ref{Ass1} is found to be equal to $\delta(x)=cx^{-2s}$.
\end{Rmk}

\subsection{Bilinear Strichartz estimates}\label{ssect bilinear}
In this subsection, we prove the bilinear estimates that will be used in the rest of this section. The proof is adapted from \cite{an} with two dimensional modification and a different counting lemma. 
\begin{Lem}[Bilinear estimates for fractional NLS]\label{lem bilinear}
For $\sigma \in [\frac{1}{2} ,1]$,  $j  =1,2$, $N_j >0$ and $u_j \in L_{rad}^2 (B^d)$ satisfying 
\begin{align*}
\mathbf{1}_{\sqrt{-\Delta} \in [N_j ,  2N_j]} u_j = u_j , 
\end{align*}
we have the following bilinear estimates. 
\begin{enumerate}
\item
The bilinear estimate without derivatives.\\
Without loss of generality, we assume $N_1 \geq N_2 $, then for any $\varepsilon >0 $
\begin{align}\label{eq bilinear1}
\norm{S_{\sigma}(t)  u_1  \, S_{\sigma}(t)  u_2}_{L_{t,x}^2 ((0,1) \times B^d)} \lesssim  N_2^{\frac{d-1}{2} + \varepsilon}   \norm{u_1}_{L_x^2 (B^d)} \norm{u_2}_{L_x^2 (B^d)} .
\end{align}

\item
The bilinear estimate with derivatives.\\
Moreover, if $u_j \in H_0^1 (B^d)$, then for any $\varepsilon >0 $
\begin{align}\label{eq bilinear2}
\norm{ \nabla S_{\sigma}(t)  u_1  \,  S_{\sigma}(t)  u_2}_{L_{t,x}^2 ((0,1) \times B^d)} \lesssim N_1N_2^{\frac{d-1}{2} +\varepsilon}  \norm{u_1}_{L_x^2 (B^d)} \norm{u_2}_{L_x^2 (B^d)}.
\end{align}
\end{enumerate}
\end{Lem}

\begin{Rmk}
Notice that in \eqref{eq bilinear1} and \eqref{eq bilinear2}, the upper bounds are independent on the fractional power $\sigma$. This is because a counting estimate in Claim \ref{claim bilinear1} does not see difference on $\sigma$. Hence the local well-posedness index is uniform for $\sigma \in [\frac{1}{2},1)$.
\end{Rmk}

\begin{Lem}[Bilinear estimates for classical NLS]\label{lem bilinear'}
Under the same setup as in Lemma \ref{lem bilinear}, the bilinear estimate analogue is given by
\begin{align*}
\norm{S_{1}(t)  u_1  \, S_{1}(t)  u_2}_{L_{t,x}^2 ((0,1) \times B^d)} & \lesssim  N_2^{\frac{d-2}{2} + \varepsilon}   \norm{u_1}_{L_x^2 (B^d)} \norm{u_2}_{L_x^2 (B^d)} \\
\norm{ \nabla S_{1}(t)  u_1  \,  S_{1}(t)  u_2}_{L_{t,x}^2 ((0,1) \times B^d)} & \lesssim N_1N_2^{\frac{d-2}{2} +\varepsilon}  \norm{u_1}_{L_x^2 (B^d)} \norm{u_2}_{L_x^2 (B^d)}.
\end{align*}
\end{Lem}


Notice that the proof of Lemma \ref{lem bilinear'} in fact can be extended from \cite{an}, hence we will only focus on the proof of Lemma \ref{lem bilinear} and the proof of the local theory for fractional NLS in the rest of this section.

\begin{Prop}[Lemma 2.3 in \cite{bgtBil}: Transfer principle]\label{prop bilinear}
For any $b > \frac{1}{2}$ and for $j =1,2$, $N_j >0$ and $f_j \in X_{\sigma}^{0,b} (\R \times B^d)$ satisfying 
\begin{align*}
\mathbf{1}_{\sqrt{-\Delta} \in [N_j , 2 N_j]} f_j = f_j ,
\end{align*}
one has the following bilinear estimates. 
\begin{enumerate}
\item
The bilinear estimate without derivatives.

Without loss of generality, we assume $N_1 \geq N_2$, then for any $\varepsilon >0 $
\begin{align}\label{eq bilinear1'}
\norm{ f_1 f_2}_{L_{t,x}^2 (\R \times B^d)} \lesssim  N_2^{\frac{d-1}{2} + \varepsilon}  \norm{f_1}_{X_{\sigma}^{0,b} (\R \times B^d)} \norm{f_2}_{X_{\sigma}^{0,b} (\R \times B^d)} .
\end{align}

\item
The bilinear estimate with derivatives.

Moreover, if $f_j \in H_0^1 (B^d)$, then for any $\varepsilon >0 $
\begin{align}\label{eq bilinear2'}
\norm{ \nabla f_1  f_2}_{L_{t,x}^2 (\R \times B^d)} \lesssim N_1  N_2^{\frac{d-1}{2} + \varepsilon}  \norm{f_1}_{X_{\sigma}^{0,b} (\R \times B^d)} \norm{f_2}_{X_{\sigma}^{0,b} (\R \times B^d)}.
\end{align}
\end{enumerate}
\end{Prop}

\begin{proof}[Proof of Lemma \ref{lem bilinear}]
First  we write
\begin{align*}
u_1 = \sum_{n_1 \sim N_1} c_{n_1}   e_{n_1}(r) ,  \quad u_2 = \sum_{n_2 \sim N_2} d_{n_2}   e_{n_2}(r) 
\end{align*}
where $c_{n_1} = (u_1 , e_{n_1})_{L^2}$ and $d_{n_2} = (u_2 , e_{n_2})_{L^2}$. Then 
\begin{align*}
S_{\sigma}(t)  u_1 = \sum_{n_1 \sim N_1} e^{-\i t z_{n_1}^{2 \sigma} } c_{n_1}  e_{n_1}(r) , \quad S_{\sigma}(t)  u_2= \sum_{n_2 \sim N_2} e^{-\i t z_{n_2}^{2 \sigma} } d_{n_2}  e_{n_2}(r) 
\end{align*}
Therefore, the bilinear objects that one needs to estimate are the $L_{t,x}^2$ norms of 
\begin{align*}
E_0(N_1, N_2) & = \sum_{n_1 \sim N_1} \sum_{n_2 \sim N_2}  e^{-\i t  (z_{n_1}^{2\sigma} +z_{n_2}^{2\sigma}) }  (c_{n_1} d_{n_2}) ( e_{n_1} e_{n_2}) ,\\
E_1(N_1, N_2) & = \sum_{n_1 \sim N_1} \sum_{n_2 \sim  N_2}  e^{-\i t  (z_{n_1}^{2\sigma} +z_{n_2}^{2\sigma}) }  (c_{n_1} d_{n_2} )(\nabla e_{n_1}  e_{n_2} ).
\end{align*}
Let us focus on \eqref{eq bilinear1} first. 
\begin{align}\label{eq bi2}
(\text{LHS  of } \eqref{eq bilinear1})^2 & = \norm{E_0(N_1, N_2) }_{L^2 ((0, 1) \times B^d)}^2  = \int_{\R \times B^d} \abs{\sum_{n_1 \sim N_1} \sum_{n_2 \sim N_2}  e^{-\i t  (z_{n_1}^{2\sigma} +z_{n_2}^{2\sigma}) }  (c_{n_1} d_{n_2}) ( e_{n_1} e_{n_2}) }^2 \, dx dt
\end{align}
Here we employ a similar argument used in the proof of Lemma 2.6 in \cite{sun2020gibbs}. We fix $\eta \in C_0^{\infty} ((0,1))$, such that $\eta \big|_{I} \equiv 1$  where $I$ is a slight enlargement of $(0,1)$. Thus we continue from \eqref{eq bi2}
\begin{align}\label{eq bi3}
\eqref{eq bi2} & \leq \int_{\R \times B^d} \eta(t) \abs{ \sum_{n_1 \sim N_1} \sum_{n_2 \sim N_2}  e^{-\i t  (z_{n_1}^{2\sigma} +z_{n_2}^{2\sigma}) }  (c_{n_1} d_{n_2}) ( e_{n_1} e_{n_2})}^2 \, dxdt \notag\\
& = \int_{\R \times B^d} \eta(t) \abs{\sum_{\tau}   \sum_{ (n_1 ,n_2) \in \Lambda_{N_1 ,N_2, \tau} }  e^{-\i t  (z_{n_1}^{2\sigma} +z_{n_1}^{2\sigma}) }  (c_{n_1} d_{n_2}) ( e_{n_1} e_{n_2})}^2 \, dxdt
\end{align}
where
\begin{align*}
\#  \Lambda_{N_1 ,N_2, \tau} = \# \{ (n_1 ,n_2) \in \N^2 : n_1 \sim N_1 , n_2 \sim N_2 , \abs{z_{n_1}^{2\sigma} +z_{n_2}^{2\sigma} - \tau}\leq \frac{1}{2} \} .
\end{align*}
By expanding the square above and using Plancherel, we have
\begin{align}\label{eq bi4}
\eqref{eq bi3} & = \int_{\R \times B^d} \eta(t) \sum_{\tau , \tau'} \sum_{\substack{ (n_1 ,n_2) \in \Lambda_{N_1 ,N_2, \tau}\\ (n_1' ,n_2') \in \Lambda_{N_1' ,N_2', \tau'} }}    e^{\i t  (z_{n_1'}^{2\sigma} +z_{n_2'}^{2\sigma} - z_{n_1}^{2\sigma} -z_{n_2}^{2\sigma}) }  (c_{n_1} d_{n_2})  (\overline{c_{n_1'} d_{n_2'}})  ( e_{n_1} e_{n_2})( e_{n_1'} e_{n_2'}) \, dxdt \notag\\
& =   \sum_{\tau , \tau'} \sum_{\substack{ (n_1 ,n_2) \in \Lambda_{N_1 ,N_2, \tau}\\ (n_1' ,n_2') \in \Lambda_{N_1' ,N_2', \tau'} }}   \widehat{\eta}((z_{n_1’}^{2\sigma} +z_{n_1'}^{2\sigma}) - (z_{n_1}^{2\sigma} +z_{n_2}^{2\sigma} ))    (c_{n_1} d_{n_2}) (\overline{c_{n_1'} d_{n_2'} }) \int_{ B^d} ( e_{n_1} e_{n_2}) ( e_{n_1'} e_{n_2'}) \, dx \notag\\
& \lesssim  \sum_{\tau ,\tau'} \frac{1}{1+ \abs{\tau -\tau'}^2}  \sum_{\substack{n_1 \sim N_1 , n_2 \sim N_2 \\n_1' \sim N_1' , n_2' \sim N_2'}}  \mathbf{1}_{\Lambda_{N_1 ,N_2, \tau}}  \mathbf{1}_{\Lambda_{N_1' ,N_2', \tau'}} (c_{n_1} d_{n_2}) (\overline{c_{n_1'} d_{n_2'}})  \norm{ e_{n_1} e_{n_2}}_{L^2 (B^d)}  \norm{ e_{n_1'} e_{n_2'}}_{L^2 (B^d)}
\end{align}
Then by Schur's test, we arrive at
\begin{align}\label{eq bi5}
\eqref{eq bi4} & \lesssim  \sum_{\tau \in \N} \parenthese{ \sum_{(n_1 ,n_2) \in \Lambda_{N_1, N_2, \tau}}  \abs{c_{n_1} d_{n_2}}  \norm{ e_{n_1} e_{n_2}}_{L^2 (B^d)} }^2 \notag\\
& \lesssim    \sum_{\tau \in \N}  \# \Lambda_{N_1 ,N_2, \tau}  \sum_{(n_1 ,n_2) \in \Lambda_{N_1, N_2, \tau}}  \abs{c_{n_1} d_{n_2}}^2  \norm{ e_{n_1} e_{n_2}}_{L^2 (B^d)}^2
\end{align}

We claim that
\begin{claim}\label{claim bilinear1}
\begin{enumerate}
\item
$\# \Lambda_{N_1, N_2, \tau}  = \mathcal{O}(N_2)$ ;
\item
$\norm{ e_{n_1} e_{n_2} }_{L^2(B^d)}^2 \lesssim  N_2^{d-2+}$ .
\end{enumerate}
\end{claim}

Assuming Claim \ref{claim bilinear1}, we see that
\begin{align*}
\eqref{eq bi5} & \lesssim   \sum_{\tau \in \N}   N_2^{d-1+}  \sum_{(n_1 ,n_2) \in \Lambda_{N_1, N_2, \tau}}  \abs{c_{n_1} d_{n_2}}^2  \lesssim N_2^{d-1 + \varepsilon}  \norm{u_1}_{L^2 (B^d)}^2 \norm{u_2}_{L^2 (B^d)}^2.
\end{align*}
Therefore, \eqref{eq bilinear1} follows.

Now we are left to prove Claim \ref{claim bilinear1}. 
\begin{proof}[Proof of Claim \ref{claim bilinear1}]
In fact, {(2)} is due to H\"older inequality and the logarithmic bound on the $L^p$ norm of $e_n$ in \eqref{eq e_n bdd}
\begin{align*}
\norm{ e_{n_1} e_{n_2} }_{L^2(B^d)} \lesssim \norm{ e_{n_1}  }_{L^{\frac{2d}{d-1} -}(B^d)} \norm{ e_{n_2} }_{L^{2d +} (B^d)} \lesssim n_2^{\frac{d-1}{2} - \frac{d}{2d} +} = n_2^{\frac{d-2}{2} +} .
\end{align*}

For {(1)}, we have that for fixed $\tau \in \N$ and fixed $n_2 \sim N_2$
\begin{align*}
\abs{z_{n_1}^{2\sigma} +z_{n_2}^{2\sigma} - \tau}\leq \frac{1}{2}  \implies z_{n_1} \in [(\tau -\frac{1}{2} - z_{n_2}^{2\sigma})^{\frac{1}{2\sigma}}  , (\tau + \frac{1}{2} - z_{n_2}^{2\sigma})^{\frac{1}{2\sigma}} ]
\end{align*}
There are at most 1 integer $z_{n_1}$ in this interval, since by convexity
\begin{align*}
(\tau + \frac{1}{2} - z_{n_2}^{2\sigma})^{\frac{1}{2\sigma}} - (\tau -\frac{1}{2} - z_{n_2}^{2\sigma})^{\frac{1}{2\sigma}}  \leq 1^{\frac{1}{2\sigma}}  =1
\end{align*}
Then
\begin{align*}
\#  \Lambda_{N_1 ,N_2, \tau} = \# \{ (n_1 ,n_2) \in \N^2 : n_1 \sim N_1 , n_2 \sim N_2 , \abs{z_{n_1}^{2\sigma} +z_{n_2}^{2\sigma} - \tau}\leq \frac{1}{2} \} \sim \mathcal{O} (N_2)
\end{align*}
We finish the proof of Claim \ref{claim bilinear1}.
\end{proof}

The estimation of  \eqref{eq bilinear2} is similar, hence omitted. 

The proof of Lemma \ref{lem bilinear} is complete now.
\end{proof}

\subsection{Nonlinear estimates}\label{subsec nonlinear}
In the local theory, we need a nonlinear estimate of the following form
\begin{align*}
\int_{\R} \int_{\Theta} u_0 u_1 u_2  u_{3} \, dx dt \leq c \norm{u_0}_{X_{\sigma}^{-s,b'} (\R \times B^d)} \prod_{j=1}^{3} \norm{u_j}_{X_{\sigma}^{s,b'} (\R \times B^d)} ,
\end{align*}
where $u_j \in \{ u, \bar{u} \}$ for $j = 1 , 2, 3$. 

To this end, we first study the nonlinear behavior of all frequency localized $u_j$'s based on the bilinear estimates that we obtained, then sum over all frequencies. 

For $j \in \{ 0, 1,2, 3\}$, let $N_j = 2^k, k \in \N$. We denote 
\begin{align*}
u_j = \sum_{N_j \leq \inner{z_n} < 2N_j} P_{z_n} u_j.
\end{align*}
Also we denote by $\underline{N} = (N_0, N_1 , N_2, N_3)$ the quadruple of $2^k$ numbers, $k \in \N$, and
\begin{align*}
L(\underline{N}) = \int_{\R \times B^d} \prod_{j=0}^{3} u_{j} \, dx dt .
\end{align*}

\begin{Lem}[Localized nonlinear estimates]\label{lem nonlinear est}
Assume that $N_1 \geq N_2 \geq N_3 $.  Then there exists $0 < b' < \frac{1}{2}$ such that one has
\begin{align}
\abs{L(\underline{N})} & \lesssim  (N_2 N_3)^{\frac{d-1}{2} + \varepsilon} \prod_{j=0}^{3} \norm{u_{j}}_{X_{\sigma}^{0, b'}} , \label{eq nonlinear est1} \\
\abs{L(\underline{N})} & \lesssim (\frac{N_1}{N_0})^2 (N_2 N_3)^{ \frac{d-1}{2} + \varepsilon} \prod_{j=0}^{3}  \norm{u_{j}}_{X_{\sigma}^{0, b'}} . \label{eq nonlinear est2}
\end{align}
\end{Lem}

\begin{Rmk}
Lemma \ref{lem nonlinear est} will play an important role in the local theory. Moreover, the first estimate \eqref{eq nonlinear est1} will be used in the case $ N_0 \leq c N_1$ while the second one will be used in the case $N_0 \geq c N_1 $. The proof of Lemma \ref{lem nonlinear est} is adapted from Lemma 2.5 in \cite{an}. We briefly present the proof, since a treatment in this proof will be used in the next section.
\end{Rmk}

\begin{proof}[Proof of Lemma \ref{lem nonlinear est}]
We start with \eqref{eq nonlinear est1}. 

On one hand, by H\"older inequality, Bernstein inequality and Lemma \ref{lem X property1}, we write
\begin{align*}
\abs{L(\underline{N})} & \lesssim \norm{u_{0}}_{L_t^{4} L_x^2} \norm{u_{1}}_{L_t^{4} L_x^2} \norm{u_{2}}_{L_t^{4} L_x^{\infty}} \norm{u_{3}}_{L_t^{4} L_x^{\infty}}   \lesssim (N_2 N_{3})^{\frac{d}{2}} \prod_{j=0}^{3} \norm{u_{j}}_{X_{\sigma}^{0, \frac{1}{4}}} .
\end{align*}
On the other hand, we can estimate $\abs{L(\underline{N})} $ using Proposition  \ref{prop bilinear}. That is,
\begin{align*}
\abs{L(\underline{N})} & \lesssim \norm{u_{0}  u_{2} }_{L_{t,x}^2 } \norm{u_{1}  u_{3} }_{L_{t,x}^2 } \lesssim (N_2 N_3)^{\frac{d-1}{2} +\varepsilon}  \prod_{j=0}^{3} \norm{u_{j}}_{X_{\sigma}^{0, b_0}}  ,
\end{align*}
where $b_0 > \frac{1}{2}$.

Interpolation between the two estimates above implies
\begin{align*}
\abs{L(\underline{N})} & \lesssim (N_2 N_3)^{ \frac{d-1}{2} + \varepsilon'}  \prod_{j=0}^{3} \norm{u_{j}}_{X_{\sigma}^{0, b'}} ,
\end{align*}
where $b' \in (0, \frac{1}{2})$ and  $\varepsilon' $ is a small positive power after interpolation. This finishes the computation of \eqref{eq nonlinear est1}.

For \eqref{eq nonlinear est2}, we will use a trick to introduce a Laplacian operator into the integral. This is the treatment that we mentioned earlier that will be used in the next section. 

First recall Green's theorem,
\begin{align*}
\int_{B^d} \Delta f g - f \Delta g \, dx = \int_{\mathbb{S}^{d-1}} \frac{\partial f}{\partial v} g - f \frac{\partial g}{\partial v} \, d \sigma .
\end{align*}
Note that
\begin{align*}
-\Delta e_k = z_k^2 e_k ,
\end{align*}
where $ z_k^2$'s are the eigenvalues defined in \eqref{eq z_n}. Then we write
\begin{align*}
u_{0} = -\frac{\Delta}{N_0^2} \sum_{z_{n_0} \sim N_0} c_{n_0} (\frac{N_0}{z_{n_0}})^2 e_{n_0} .
\end{align*}
Define
\begin{align*}
T u_{0 } & = \sum_{z_{n_0} \sim N_0} c_{n_0} (\frac{N_0}{z_{n_0}})^2 e_{n_0} , \qquad V u_{0}  = \sum_{z_{n_0} \sim N_0} c_{n_0} (\frac{z_{n_0}}{N_0})^2 e_{n_0} .
\end{align*}
It is easy to see that for all $s$
\begin{align*}
TV  u_{0} & = VT u_{0} = u_{0} ,\\
\norm{T u_{0} }_{H_x^s} & \sim \norm{ u_{0} }_{H_x^s} \sim \norm{V u_{0} }_{H_x^s} .
\end{align*} 
Using this notation, we write
\begin{align*}
u_{0} = -\frac{\Delta}{N_0^2} T u_{0} 
\end{align*}
and
\begin{align*}
L(\underline{N}) = \frac{1}{N_0^2} \int_{\R \times B^d} T u_{0} \Delta (\prod_{j=1}^{3} u_{j}) .
\end{align*}
By the product rule and the assumption that $N_1 \geq N_2 \geq N_3$, we only need to consider the two largest cases of $\Delta ( u_1  u_{2} u_{3}) $. They are
\begin{enumerate}
\item
$(\Delta u_{1}) u_{2} u_{3}$
\item
$(\nabla u_{1}) \cdot (\nabla u_{2})  u_{3} $ .
\end{enumerate}
We denote
\begin{align*}
J_{11} (\underline{N}) & = \int_{\R \times B^d} T u_{0} (\Delta u_{1}) u_{2} u_{3} , \\
J_{12} (\underline{N}) & =\int_{\R \times B^d} T u_{0} (\nabla u_{1}) \cdot (\nabla u_{2})u_{3}  .
\end{align*}
Using $\Delta u_i = -N_i^2 V u_i$, we obtain
\begin{align*}
\frac{1}{N_0^2} \abs{J_{11} (\underline{N})} \lesssim (\frac{N_1}{N_0})^2 (N_2 N_3)^{\frac{d-1}{2} + \varepsilon}  \prod_{j=0}^{3} \norm{u_{j}}_{X_{\sigma}^{0, b'}} .
\end{align*}
Now for $\abs{J_{12} (\underline{N})} $, we estimate it in a similar fashion that we did in \eqref{eq nonlinear est1}. On one hand, by H\"older inequality, Bernstein inequality and Lemma \ref{lem X property1}, we have
\begin{align*}
\abs{J_{12} (\underline{N})} & \lesssim N_1 N_2 (N_2 N_{3})^{\frac{d}{2}} \prod_{j=0}^{3} \norm{u_{j}}_{X_{\sigma}^{0, \frac{1}{4}}} .
\end{align*}
On the other hand, using Proposition \ref{prop bilinear}, we get
\begin{align*}
\abs{J_{12} (\underline{N})} & \lesssim \norm{\nabla u_{1}  u_{3} }_{L_{t,x}^2 } \norm{ \nabla u_{2} Tu_{0}  }_{L_{t,x}^2 }  \lesssim N_1 N_2 (N_2 N_3)^{\frac{d-1}{2} + \varepsilon}  \prod_{j=0}^{3} \norm{u_{j}}_{X_{\sigma}^{0, b_0}} .
\end{align*}
Interpolation between the two estimates above implies
\begin{align*}
\frac{1}{N_0^2} \abs{J_{12} (\underline{N})} \lesssim  \frac{N_1 N_2}{N_0^2} (N_2 N_3)^{\frac{d-1}{2} + \varepsilon'}  \prod_{j=0}^{3} \norm{u_{j}}_{X_{\sigma}^{0, b'}}. 
\end{align*}
Similarly  $\varepsilon'$ is a small positive power after interpolation.
The proof of Lemma \ref{lem nonlinear est} is complete.
\end{proof}

\begin{Prop}[Nonlinear estimates]\label{prop nonlinear est}
For $\sigma \in [\frac{1}{2} ,1]$, $s > s_{l}(\sigma)$ as in \eqref{eq s*}, there exist $b, b' \in \R$ satisfying
\begin{align*}
0 < b' < \frac{1}{2} < b, \quad b + b' < 1 ,
\end{align*}
such that for every triple $(u_1, u_2 , u_3)$ in $X_{\sigma}^{s,b} (\R \times B^d)$,
\begin{align*}
\norm{\abs{u}^2 u}_{X_{\sigma}^{s, -b'} (\R \times B^d)} \lesssim   \prod_{j=1}^3 \norm{u}_{X_{\sigma}^{s,b} (\R \times B^d)}^3  .
\end{align*} 
\end{Prop}

\begin{proof}[Proof of Proposition \ref{prop nonlinear est}]
Based on Lemma \ref{lem nonlinear est}, we only need to consider $L= \sum_{\underline{N}} L(\underline{N})$. By symmetry, we can reduce the sum into the following two cases:
\begin{enumerate}
\item
$ N_0 \leq c N_1$
\item
$ N_0 \geq c N_1$.
\end{enumerate}

Case 1: $ N_0 \leq c N_1$.

Using Lemma \ref{lem nonlinear est} and Cauchy–Schwarz inequality, we obtain
\begin{align*}
\abs{\sum L(\underline{N})} & \lesssim \sum_{ \substack{ N_1 \geq N_2 \geq N_3 \\ N_0 \leq c N_1}}   (N_2 N_3)^{\frac{d-1}{2} + \varepsilon}  \norm{v_{0}}_{X_{\sigma}^{0, b'}}  \prod_{j=1}^{3} \norm{u_{j}}_{X_{\sigma}^{0, b'}} \\
& \lesssim \sum_{ \substack{ N_1 \geq N_2 \geq N_3  \\ N_0 \leq c N_1}}   \frac{N_0^s}{N_1^s}  (N_2 N_3)^{-s + \frac{d-1}{2}  + \varepsilon}  \norm{v_{0}}_{X_{\sigma}^{-s,b'}} \prod_{j=1}^{2q+1} \norm{u_{j}}_{X_{\sigma}^{s, b'}} \\
& \lesssim \norm{v}_{X_{\sigma}^{-s,b'}}  \norm{u}_{X_{\sigma}^{s, b'}}^3   ,
\end{align*}
where $s > s_l(\sigma) $. For the last inequality, we sum from the smallest index $N_{3}$ to the largest one  using Cauchy–Schwarz inequality. Then by the embedding $X_{\sigma}^{s,b} \subset  X_{\sigma}^{s , b' }  $ and duality, we have the estimate in Case 1.

Case 2: $ N_0 \geq c N_1$.

Similarly, by Lemma \ref{lem nonlinear est} and Cauchy–Schwarz inequality again, we have
\begin{align*}
\abs{\sum L(\underline{N})} & \lesssim \sum_{ \substack{ N_1 \geq N_2 \geq N_3  \\ N_0 \geq c N_1}}    (\frac{N_1}{N_0})^2 (N_2 N_3)^{\frac{d-1}{2} + \varepsilon} \norm{v_0}_{X_{\sigma}^{0, b'}}   \prod_{j=1}^{3} \norm{u_{j}}_{X_{\sigma}^{0, b'}} \\
& \lesssim \sum_{ \substack{ N_1 \geq N_2 \geq N_3  \\ N_0 \geq c N_1}}   (\frac{N_1}{N_0})^{2-s} \frac{N_0^s}{N_1^s}  (N_2 N_3)^{-s + \frac{d-1}{2} + \varepsilon}  \norm{v_{0}}_{X_{\sigma}^{-s,b'}} \prod_{j=1}^{3} \norm{u_{j}}_{X_{\sigma}^{s, b'}} \\
& \lesssim \norm{v_0}_{X_{\sigma}^{-s,b'}}  \norm{u}_{X_{\sigma}^{s, b'}}^3  ,
\end{align*}
where $s >  s_l(\sigma)$. Then by the embedding  $X_{\sigma}^{s,b} \subset  X_{\sigma}^{s , b' }  $  and duality, we have the estimate in Case 2.
The proof of Proposition \ref{prop nonlinear est} is complete.
\end{proof}

\subsection{Local well-posedness}
We now turn to the proof of the main result in this section:  local well-posedness.
\begin{proof}[Proof of Theorem \ref{thm LWP}]
Let $u_0 \in H^{s}$. We first define a map
\begin{align*}
F (u)(t) : = S_{\sigma}(t)  u_0  - \i \int_0^t S_\sigma(t-s)  \abs{u}^{2} u \, ds.
\end{align*}
We prove  this locally well-posedness theory using a standard fixed point argument. 
Let $R_0>0$ and $u_0 \in H^s (B^d)$ with $\norm{u_0}_{H^s} \leq R_0$. We show that there exists $R>0 $ and $0 < \delta= \delta(R_0) <1 $ such that $F$ is a contraction mapping from $B(0, R) \subset X_{\sigma, \delta}^{s,b} (B^d)$ onto itself.

Define $R = 2 c_0 R_0$. 
\begin{enumerate}
\item $F$ is a self map from $B(0, R) \subset X_{\sigma,\delta}^{s,b} (B^d)$ onto itself. 

For $\delta< 1$, by Lemma \ref{lem Duhamel}, Proposition \ref{prop nonlinear est} and Sobolev embedding
\begin{align*}
\norm{F (u)}_{X_{\sigma,\delta}^{s,b} (B^d)} & \leq c_0 \norm{u_0}_{H^s} + c_1 \delta^{1- b - b'} \norm{\abs{u}^{2} u}_{X_{\sigma,\delta}^{s, -b'}(B^d)} \\
& \leq  c_0 \norm{u_0}_{H^s} + c_2 \delta^{1- b - b'} \delta^{3(b-b')} \norm{u}_{X_{\sigma,\delta}^{s, b}(B^d)}^{3} \\
& \leq c_0 \norm{u_0}_{H^s} + c_2 \delta^{1 +2 b - 4b'} R^{3}.
\end{align*}
Taking $\delta_1 =  (\frac{c_0}{c_2 R^{2}})^{\frac{1}{1 +2 b - 4b'}} <1$ such that $ c_2 \delta_1^{1 +2 b - 4b'} R^{3} = c_0 R$, we see that $F$ is a self map.

\item $F$ is a contraction mapping.

Similarly, by Lemma \ref{lem Duhamel} and Proposition \ref{prop nonlinear est}
\begin{align}\label{eq LWP1}
\norm{F (u) - F (v)}_{X_{\sigma,\delta}^{s,b}(B^d)} & \leq c_3 \delta^{1 +2 b - 4b'} (\norm{u}_{X_{\sigma,\delta}^{s, b}(B^d)}^{2} + \norm{v}_{X_{\sigma,\delta}^{s, b}(B^d)}^{2}) \norm{u-v}_{X_{\sigma,\delta}^{s, b}(B^d)} \notag \\
& \leq c_4 \delta^{1 +2 b - 4b'} R^{2} \norm{u-v}_{X_{\sigma,\delta}^{s, b}(B^d)} .  
\end{align}
Taking $\delta_2 = (\frac{1}{2c_4 R^{2}})^{\frac{1}{1 +2 b - 4b'}}$ such that $c_4 \delta_2^{1 +2 b - 4b'} R^{2} = \frac{1}{2}$, we can make $F$ a contraction mapping.

Now we choose 
\begin{align}\label{eq T_R}
\delta_R = \min\{\delta_1, \delta_2 \}.
\end{align}
As a consequence of the fixed point argument, we have
\begin{align*}
\norm{u}_{X_{\sigma,\delta}^{s,b}(B^d)} \leq 2 c_0 \norm{u_0}_{H^s} . 
\end{align*}

\item Stability.

If $u$ and $v$ are two solutions to \eqref{fNLS} with initial data $u(0) =u_0$ and $v(0) = v_0$ respectively. Then by \eqref{eq LWP1} and the choice of $\delta_R$ \eqref{eq T_R}
\begin{align*}
\norm{u-v}_{X_{\sigma,\delta}^{s,b}(B^d)} & \leq c_0 \norm{u_0 -  v_0}_{ H^s} + c_4 \delta^{1 +2 b - 4b'} R^{2}  \norm{u-v}_{X_{\sigma,\delta}^{s, b}(B^d)} ,
\end{align*}
which implies
\begin{align*}
\norm{u-v}_{X_{\sigma,\delta}^{s, b}(B^d)} \leq c \norm{u_0 -  v_0}_{H^s} . 
\end{align*}
\end{enumerate}
Now the proof of Theorem \ref{thm LWP} is finished. 
\end{proof}

\begin{Rmk}
For finite $N$ let us denote $\Pi^N L^2$ by $E_N$ and the local flow on $E_N = \Pi^N L^2$ by $\phi_t^N$. We verify easily that the local existence time obtained in Theorem \ref{thm LWP} is valid for the Galerkin approximations for the same radius $R$ of the balls. We set $\phi_t$ to be the local flow constructed in Theorem \ref{thm LWP} in $X_{\sigma}^{s,b}$. 
\end{Rmk}

\begin{Prop}[Convergence of Galerkin projections to FNLS]\label{Prop:CvgLocSol}
Let $s>s_l(\sigma)$ as in \eqref{eq s*}, $u_0 \in H^{s}$ and $(u_0^N) $ be a sequence that converges to $u_0 $ in $H^{s}$, where $u_0^N \in E_N$. Then for any $r\in(s_l(\sigma),s)$
\begin{align*}
\norm{\phi_t^N u_0^N - \phi_t u_0}_{X^{r,b}_{\sigma, T}} \to 0 \quad \text{ as } N \to \infty.
\end{align*} 
\end{Prop}

\begin{proof}[Proof of Proposition \ref{Prop:CvgLocSol}]
Since $u_0^N \to u_0$, there exists an $R$ such that $u_0^N \in B_R(H^{\gamma})$. Take the same $T_R$ as in \eqref{eq T_R}. Set $w_0^N = u_0^N - u_0$ and $w^N = \phi_t^N u_0^N - \phi_t u_0$, then we write
\begin{align*}
w^N & = S_\sigma(t) w_0^N - \i \int_0^t  S_\sigma(t-s) (\Pi^N \abs{\phi_t^N u_0^N}^{2}\phi_t^N u_0^N - \abs{\phi_t u_0}^{2}\phi_t u_0 ) \, ds\\
& = S_\sigma(t) w_0^N -  \i \int_0^t  S_\sigma(t-s) \Pi^N \parenthese{ \abs{\phi_t^N u_0^N}^{2}\phi_t^N u_0^N  - \abs{\phi_t u_0}^{2}\phi_t u_0} - (1-\Pi^N ) (\abs{\phi_t u_0}^{2}\phi_t u_0 ) \, ds .
\end{align*}
Using H\"older inequality, Lemma \ref{lem Duhamel}, Proposition \ref{prop nonlinear est} and the choice of $T_R$ as in \eqref{eq T_R}, we have
\begin{align*}
\norm{w^N}_{X^{r,b}_{\sigma , T}} & \leq \norm{w_0^N}_{H^s} + C \norm{w^N}_{X^{r,b}_{\sigma , T}} (\norm{u^N}_{X^{r,b}_{\sigma , T}}^{2} + \norm{u}_{X^{r,b}_{\sigma , T}}^{2} ) +   \norm{(1-\Pi^N) \abs{u}^{2} u}_{X^{r,b}_{\sigma , T}} \\
& \leq \norm{w_0^N}_{H^s} + C\norm{w^N}_{X^{r,b}_{\sigma , T}} (\norm{u^N}_{X^{r,b}_{\sigma , T}}^{2} + \norm{u}_{X^{r,b}_{\sigma , T}}^{2} ) + C_1  z_N^{\frac{r-s}{2}} \norm{u}_{X^{s,b}_{\sigma , T}}^{3} .
\end{align*}
Therefore,
\begin{align*}
\norm{w^N}_{X^{r,b}_{\sigma , T}} & \leq \norm{w_0^N}_{H^s} + \frac{1}{2} \norm{w^N}_{X^{r,b}_{\sigma , T}} + C_2 z_N^{\frac{\sigma-s}{2}}\\
\norm{w^N}_{X^{r,b}_{\sigma , T}} & \leq 2  \norm{w_0^N}_{H^s} + 2C_2 z_N^{\frac{r-s}{2}}.
\end{align*}
Recall that $z_N^2 \sim N^2$, then 
\begin{align*}
\norm{w^N}_{X^{r,b}_{\sigma , T}} \to 0 \quad \text{ as } N \to \infty.
\end{align*}
Hence the proof of Proposition \ref{Prop:CvgLocSol} is finished.
\end{proof}
\begin{Rmk}
In Section \ref{Sect:Pr as GWP 1}, we fulfilled Assumptions \ref{Ass:2nd:CL} to \ref{Ass:unif:conv} for any $\sigma\in (0,1]$. However, Theorem \ref{thm LWP} covers only the range $\sigma\in [\frac{1}{2},1]$ and the given regularities. 
The case $\sigma\in(0,\frac{1}{2})$ was not treated in Theorem \ref{thm LWP} because of the convexity needed in the proof of Claim \ref{claim bilinear1}. Nevertheless, in smooth regularities $s>\frac{d}{2}$ a same local well-posedness as in Theorem \ref{thm LWP} for $\sigma\in(0,\frac{1}{2})$ holds true. We will not present the details of such easy argument. We also readily have  the convergence established in Proposition \eqref{Prop:CvgLocSol} for such values of $\sigma$. This ends the proof of Theorem \ref{asGWP:1}.
\end{Rmk}

\section{Proof of Theorem \ref{asGWP:2}}\label{Sect:Pr as GWP2}
In this section we present the proof of Theorem \ref{asGWP:2}. We consider then $\sigma\in (0,1]$ and $\max(\sigma,\frac{1}{2})\leq s\leq 1+\sigma$, or $\sigma\in(\frac{1}{2},1]$ and $s\in(0,\sigma]$, with $d\geq 2$.
Recall that from \eqref{Est:Ener:low}
\begin{align}
\int_{L^2}\|u\|_{H^s}^2+\||u|^2\|_{\dot{H}^{\sigma}}^2+\||u|^2u\|_{L^2}^2\mu^N(du)\leq C,\label{Est:Recap}
\end{align}
where $C$ does not depend on $N$.
\subsection{Existence of global solutions and invariance of the measure}\label{Subsection:Prob}
For an arbitrary positive integer $k$, let us consider the spaces
\begin{align}
X_k &=L^2([0,k],H^s)\cap(H^1([0,k],H^{s-2\sigma})+H^1([0,k],L^2))=:X_k^1+X_k^2;\\
Y_k &= L^2([0,k],H^{s-})\cap  C([0,k],H^{(s-2\sigma)-}).
\end{align}
Combining standard embedding results we obtain the  compact embedding $X_k\subset\subset Y_k$.

Denote by $\nu^N_k$ the laws of the processes $(u^N(t))_{t\in [0,k)}$ that are seen as random variables valued in $C([0,k],H^{(s-2\sigma)-})$. We see, using the invariance, the relation between $\mu^N$ and $\nu^N_k$: 
\begin{align}
\mu^N =\nu^N_k\Big|_{t=t_0}\quad \text{for any $t_0\in [0,k]$}.\label{Ms:Rest:ppt:N}
\end{align}

\begin{Prop}\label{Prop:ST:bound}
We have the estimate
\begin{align}
\int_{X_k}\|u\|_{X_k}^2\nu^N_k(du)\leq C,
\end{align}
where $C$ does not depend on $N$.
\end{Prop}
\begin{proof}[Proof of Proposition \ref{Prop:ST:bound}]
Let us write the fractional NLS \eqref{fNLS-Glk} in the integral form
\begin{align}
u^N(t)=u^N(0)-\i\int_0^t(-\Delta)^\sigma u^Nds-\i\int_0^t\Pi^N|u^N|^2u^Nds.
\end{align}
Now using \eqref{Est:Recap} we have that
\begin{align}
\E\|\int_0^t(-\Delta)^\sigma u^Nds\|^2_{H^1([0,k],H^{s-2\sigma})} &\leq \E\int_0^k\|u^N\|^2_{H^s}ds\leq Ck;\\
\E\|\int_0^t\Pi^N|u^N|^2u^Nds\|^2_{H^1([0,k],L^2)} &\leq \E\int_0^k\|\Pi^N|u^N|^2u^N\|^2_{L^2}ds\leq Ck.
\end{align}
We recall that,
\begin{align}
\E\int_0^k\|u^N\|_{H^s}^2ds\leq Ck.
\end{align}
Combining the above estimates we obtain Proposition \ref{Prop:ST:bound}.
\end{proof}

\begin{Prop}\label{Prop:Compact}
$u^N$ is compact in $Y_k$ $\P-$almost surely.
\end{Prop}
\begin{proof}[Proof of Proposition \ref{Prop:Compact}]
Combining Proposition \eqref{Prop:ST:bound} and the Prokhorov theorem (Theorem \ref{Thm:Prokh}) we obtain a weak compactness of $(\nu^N_k)$.
Thanks to the Skorokhod representation theorem (see Theorem \ref{Thm:Skorokhod}), there are random variables $\tilde{u}_k^N$ and $\tilde{u}_k$, defined on a same probability space that is denoted by $(\Omega,\P)$, such that
\begin{enumerate}
\item $\tilde{u}_k^N\to \tilde{u_k}$ weakly on $X_k$,
\item $\tilde{u}_k^N$ and $\tilde{u}_k$ are distributed by $\nu_k^N$ and $\nu_k$ respectively. 
\end{enumerate} 
The proof of Proposition \ref{Prop:Compact} is finished.
\end{proof}

\begin{Prop}\label{Prop:Conv:nl}
$\Pi^N|u^N|^2u^N$ converges to $|u|^2u$ in $L^1_{t}H^{-\frac{d}{2}-1}$ as $N\to\infty$. 
\end{Prop}
\begin{enumerate}
\item $u^N$ converges to $u$ almost surely in $L^2_tH^{s-}\subset L^2_{t,x}$
\item 
$|u^N|^2$ converges almost surely in $L^2_tH^{\sigma-}\subset L^2_{t,x}$. In particular $|u^N|^2$ is bounded in $L^2_{t,x}$ almost surely, i.e $\|u^N\|_{L^4_{t,x}}$ is bounded $\P-$almost surely.
\end{enumerate}
Now let us write, by using $\Pi^N|u^N|^2u^N=|u^N|^2u^N-\Pi^{> N}|u^N|^2u^N$, a Bernstein inequality and convergence and boundedness established above,
\begin{align}
\|\Pi^N|u^N|^2u^N-|u|^2u\|_{L^1_{t}H^{-\frac{d}{2}-1}} &\leq C\||u^N|^2u^N-|u|^2u\|_{L^1_{t,x}}+\|\Pi^{> N}|u^N|^2u^N\|_{L^1_{t}H^{-\frac{d}{2}-1}}\\
&\leq \|u^N-u\|_{L^2_{t,x}}(\||u^N|^2\|_{L^2_{t,x}}+\||u|^2\|_{L^2_{t,x}})+N^{-1}\||u^N|^2u^N\|_{L^1_tH^{-\frac{d}{2}}}\\
&\leq \tilde{C}(\omega)\|u^N-u\|_{L^2_{t,x}}+C(k)N^{-1}\||u^N|^2u^N\|_{L^\frac{4}{3}_{t}L^{\frac{4}{3}}_x}\\
&\lesssim_{\omega,k} \|u^N-u\|_{L^2_{t,x}}+N^{-1}\|u^N\|_{L^4_{t,x}}^3\\
&\lesssim_{\omega,k} \|u^N-u\|_{L^2_{t,x}}+N^{-1}.
\end{align}
Combining Propositions \ref{Prop:Compact} and \ref{Prop:Conv:nl} with a diagonal argument we obtain the almost sure existence of global solutions for FNLS with $\sigma\in (0,1)$, $0< s\leq 1+\sigma$ and $d\geq 2$.

For the invariance of the law of the constructed solution $u(t)$, let us denote by $\nu$ the law of the process $u=(u(t))_{t\in\R}$. From the subsequence in the $N$-parameter that produced the limiting measure $\nu$, we can extract a subsequence, using the Prokhorov theorem and \eqref{Est:Recap}, that produces a measure $\mu$ as a weak limit point of $(\mu^N={\nu^N}{\Big|_{t=t_0}})$. Passing to the limit along this subsequence in the relation \eqref{Ms:Rest:ppt:N}, we see that $\mu =\nu\Big|_{t=t_0}$ for any $t_0\in\R$. This implies that $\mu$ is an invariant law for $u$.
\subsection{Uniqueness and Continuity}
The uniqueness argument is formulated in the following observation.
\begin{Prop}\label{prop 6.4}
The space $\Lambda= C_tL^2_x \cap L^2_tH^{\frac{1+}{2}}(B^d)$ is a uniqueness class for the fractional $NLS$ \eqref{fNLS}.
\end{Prop}
\begin{proof}[Proof of Proposition \ref{prop 6.4}]
Let  $u,\ v\in\Lambda$ be two solutions to \eqref{fNLS}, both starting at $u_0$. Let us write the Duhamel formulation
\begin{align}
u(t)=S_{\sigma}(t) u_0-\i\int_0^tS_{\sigma}(t-s)|u(s)|^2u(s)d\tau,
\end{align}
and the same for $v$.\\
Set $w=u-v$ and $\chi_\epsilon=1_{\{r\geq \epsilon\}}$.
\begin{align}
\|w\chi_\epsilon\|_{L^2}\lleq \int_0^t\|(|u|^2+|v|^2)w\chi_\epsilon\|_{L^2}d\tau.
\end{align}
Now using the radial Sobolev inequality \eqref{Radial:Sob}, we obtain
\begin{align}
|u(r)|^2+|v(r)|^2\lleq r^{1^+-d}(\|u\|^2_{H^s}+\|v\|^2_{H^s}) \quad r>0.
\end{align}
Hence
\begin{align}
\|w\chi_\epsilon\|_{L^2}\lleq \epsilon^{1^+-d}\int_0^t(\|u\|_{H^s}^2+\|v\|_{H^s}^2)\|w\chi_\epsilon\|_{L^2}d\tau.
\end{align}
It follows from the Gronwall that $\|w\chi_\epsilon\|_{L^2}=0$.
Now, write
\begin{align}
\|w\|_{L^2}\leq \|w\chi_\epsilon\|_{L^2}+\|w(1-\chi_\epsilon)\|_{L^2}\leq \|w\|_{L^4}\|1-\chi_\epsilon\|_{L^4}.
\end{align}
First, since $u,v\in\Lambda_s$, we have that $w^2\in L^2$ almost surely in the $t$ variable (hence $w\in L^4$ almost surely in $t$). Second, the Lebesgue dominated convergence theorem shows that $\|1-\chi_\epsilon\|_{L^4}\to 0$ as $\epsilon\to 0$. We obtain $\|w\|_{L^2}=0$ almost surely in $t$. Now by continuity in $t$, we have that $\|w\|_{L^2}=0$ for all $t$.Then the proof of Proposition \ref{prop 6.4} is finished.
\end{proof}

As for the continuity, let us present the following result:
\begin{Prop}\label{prop 6.5}
For $u_0,\ v_0\in \text{supp}(\mu)$, let $u, \ v$ be the corresponding solutions to \eqref{fNLS}. Let us set $w=u-v$ and $w_0=u_0-v_0$, we have
\begin{align}
\|w\|_{L^2}^2\leq C(\||u|^2\|_{L^2_tH^\theta},\||v|^2\|_{L^2_tH^\theta})\|w_0\|_{L^2}+(\|u\|_{L^4}^2+\|v\|_{L^4}^2)\epsilon(\|w_0\|_{L^2}),
\end{align}
where $\epsilon(\|w_0\|_{L^2})\to 0$ as $\|w_0\|_{L^2}\to 0$.
\end{Prop}
\begin{proof}[Proof of Proposition \ref{prop 6.5}]
Let us take $\theta =\frac{1}{2}+$ and set 
\begin{align}
\gamma=\begin{cases}
(-\ln\|w_0\|_{L^2})^{\frac{1}{4(2\theta-d)}} &\text{if $0<\|w_0\|_{L^2}< 1$}\\
0 &\text{if $ w_0 = 0$}.
\end{cases}
\end{align}
Since $\theta<\frac{d}{2}$, we see that $\gamma\to 0$ as $\|w_0\|_{L^2}\to 0$. So we can consider $\gamma<<1$ and take, as above, the function $\chi_\gamma =1_{\{r\geq\gamma\}}$. Using the radial Sobolev inequality \eqref{Radial:Sob}, we have
\begin{align}
\|w\chi_\gamma\|_{L^2}^2\leq \|w_0\chi_\gamma\|_{L^2}^2+C \gamma^{2(2\theta-d)}\int_0^t(\|u\|_{H^\theta}^2+\|v\|_{H^\theta}^2)\|w\chi_\gamma\|_{L^2}^2d\tau.
\end{align}
By the Gronwall inequality we obtain
\begin{align}
\|w\chi_\gamma\|_{L^2}^2\leq \|w_0\chi_\gamma\|_{L^2}^2e^{\gamma^{2(2\theta-d)}C\int_0^t(\|u\|_{H^\theta}^2+\|v\|_{H^\theta}^2)d\tau} &\leq \|w_0\chi_\gamma\|_{L^2}^2e^{\gamma^{4(2\theta-d)}}e^{C^2\left(\int_0^t\|u\|_{H^\theta}^2+\|v\|_{H^\theta}^2d\tau\right)^2}\\
&\leq e^{C^2\left(\int_0^t\|u\|_{H^\theta}^2+\|v\|_{H^\theta}^2d\tau\right)^2}\|w_0\|_{L^2}.
\end{align}
Therefore
\begin{align}
\|w\|_{L^2}^2 &\leq e^{C^2\left(\int_0^t\|u\|_{H^\theta}^2+\|v\|_{H^\theta}^2d\tau\right)^2}\|w_0\|_{L^2}+\|w(1-\chi_\gamma)\|_{L^2}^2\\
&\leq e^{C^2\left(\int_0^t\|u\|_{H^\theta}^2+\|v\|_{H^\theta}^2d\tau\right)^2}\|w_0\|_{L^2}+\|1-\chi_\gamma\|_{L^4}^2\|w\|_{L^4}^2.
\end{align}
Since $\|1-\chi_\gamma\|_{L^4}^2\to 0$ as $\|w_0\|_{L^2}\to 0$ we arrive at the result of Proposition \ref{prop 6.5}. 
\end{proof}
Notice that Remark \ref{Rmk:LargeD} applies also in the present situation. This completes the proof of Theorem \ref{asGWP:2}.
\appendix

\section{Results on probability measures and stochastic processes}

In this section $(E,\|.\|)$ is a Banach space, $C_b(E)$ denotes the space of bounded continuous functions $f:E\to\R.$ We present some useful standard results from measure theory. 

\subsection{Convergence of measures}
\begin{Def}
Let $(\mu_n)_n$ be a sequence of Borel probability measures and $\mu$ a Borel probability measure on $E$. We say that $(\mu_n)_n$ converges weakly to $\mu$ if for all $f\in C_b(E)$
\begin{align*}
\lim_{n\to\infty}\int_Ef(x)\mu_n(dx)=\int_Ef(x)\mu(dx).
\end{align*}
We write $\mu_n\rightharpoonup\mu.$
\end{Def}

\begin{Def}
A family $\Lambda$ of Borel probability measures on $E$ is said to be tight if for all $\epsilon>0$ there is a compact set $K_\epsilon\subset E$ such that for all $\mu\in \Lambda$
\begin{align*}
\mu(K_\epsilon)\geq 1-\epsilon.
\end{align*} 
\end{Def}

\begin{Thm}[Portmanteau theorem, see Theorem 11.1.1 in \cite{dudley}]\label{Thm:Portm}
Let $(\mu_n)_n$ be a sequence of probability measures and $\mu$ a probability measure on $E$, the following are equivalent:
\begin{enumerate}
\item $\mu_n\rightharpoonup\mu$,
\item for all open sets $U$, $\liminf_{n\to\infty}\mu_n(U)\geq\mu(U)$,
\item for all closed sets $F$, $\limsup_{n\to\infty}\mu_n(F)\leq\mu(F)$.
\end{enumerate}
\end{Thm}

\begin{Thm}[Prokhorov theorem, see Theorem 2.3 in \cite{daz}]\label{Thm:Prokh}
A set $\Lambda$ of Borel probability measures is relatively compact in $E$ if and only if it is tight.
\end{Thm}

\begin{Thm}[Skorokhod representation theorem, see Theorem 2.4 in \cite{daz}]\label{Thm:Skorokhod}
For any sequence $(\mu_n)_n$ of Borel probability measures on $E$ converging weakly to $\mu$, there is a probability space $(\Omega_0,\mathcal{F}^0,\P_0)$, and random variables $X,\ X_1,\cdots$ on $(\Omega_0,\mathcal{F}^0,\P_0)$ such that
\begin{enumerate}
\item $\mathcal{L}(X_n)=\mu_n$ and $\mathcal{L}(X)=\mu$,
\item $\lim_{n\to\infty}X_n=X$, $\P_0-$almost surely.
\end{enumerate}
Here $\mathcal{L}(Y)$ stands for the law of the random variable $Y.$
\end{Thm}

\subsection{Stochastic processes}

Recall that if $(X_t)_t$ is an $E-$valued martingale, then since $\|.\|$ is a convex function, we have that $(\|X_t\|)_t$ is a submartingale (by Jensen inequality). Here is a version of the Doob maximal inequality (see Theorem 3.8 in \cite{KS98}): 
\begin{Thm}[Doob maximal inequality]\label{Thm:Doob}
We have that for $p>1.$
\begin{align*}
\E\left(\sup_{t\in [0,T]}\|X(t)\|\right)^p\leq \left(\frac{p}{p-1}\right)^p\E\|X(T)\|^p.
\end{align*}
\end{Thm}

The following can be found in Chapters 3 and 4 in \cite{oks}.
\begin{Def}
An $N$-dimensional It\^o process is a proces $X=(X_t)_{t}$ of the form 
\begin{align}
X(t)=X(0)+\int_0^tu(s)  \, ds+\int_0^tv(s) \, dB(s)\label{Ito:proc}
\end{align}
where $B=(B_i)_{1\leq i\leq N}$ is an $N$-dimension Brownian motion, $u=(u_i)_{1\leq i\leq N}$ is an $N$-dimensional stochastic process and $v=(v_{ij})_{1\leq i,j\leq N}$ is a $N\times N$ matrix that both are adapted with respect to $(B_t)_t$ and satisfy the following
\begin{align*}
\P\left(\int_0^t|u(s)|+|v(s)|^2 \, ds<\infty\quad \forall t>0\right) &=1.
\end{align*}
\end{Def}

A short notation for \eqref{Ito:proc} is given by
\begin{align*}
dX=u \, dt+v \, dB.
\end{align*}
\begin{Thm}[$N$-dimensional It\^o formula]\label{Thm:Ito}
Let $X$ be a $N$-dimensional It\^o process as in \eqref{Ito:proc}. Let $f:\R^N\to\R$ be a $C^2$ function, then $f(X)$ is a $1$-dimensional It\^o process and satisfy
\begin{align*}
df(X)=\nabla f(X)\cdot dX+\frac{1}{2}\sum_{i,j}\partial^2_{x_ix_j}f(X) \, dX_idX_j,
\end{align*}
with the properties $dB_i \, dB_j=\delta_{ij} \, dt$, $dt \, dB_i=0$.
\end{Thm}
Using the properties above, we can remark that in the particular case where $v$ is diagonal (as in this paper), we have
\begin{align*}
\sum_{i,j}\partial^2_{x_ix_j}f(X) \, dX_idX_j=\sum_{i}\partial^2_{x_i}f(X)v_i^2 \, dt.
\end{align*}

\section{Some important lemmas}\label{Apx B}
In this section $(P_t)_{t \geq 0}$ be a Feller semi-group (a Markov semi-group satisfying the Feller property) on a Banach space $X$, $P_t^*$ is the adjoint operator of $P_t$.
\begin{Lem}[Krylov-Bogoliubov argument]\label{Lem:KBarg}
If there exists $t_n \to \infty$ and $\mu \in \mathfrak{P} (X)$ such that $\frac{1}{t_n} \int_0^{t_n} P_{t}^* \delta_0 \, dt \rightharpoonup \mu $ in $X$, then $P_t^* \mu = \mu$ for all $t \geq 0$. $\delta_0$ is the Dirac measure at 0.
\end{Lem}

\begin{proof}[Proof of Lemma \ref{Apx B}]
Note that $P_r f \in  C_b$. We have
\begin{align*}
\langle f, P_r^* \mu\rangle & = \langle P_r f , \mu\rangle = \lim_{t_n \to \infty} \langle P_r f , \frac{1}{t_n} \int_0^{t_n} P_t^* \delta_0 \, dt\rangle  = \lim_{t_n \to \infty} \langle f , P_r^* \frac{1}{t_n} \int_0^{t_n} P_t^* \delta_0 \, dt\rangle \\
& = \lim_{t_n \to \infty} \langle f ,  \frac{1}{t_n} \int_0^{t_n} P_{t+r}^* \delta_0 \, dt\rangle  = \lim_{t_n \to \infty} \langle f ,  \frac{1}{t_n} \int_{r}^{t_n+r} P_{t}^* \delta_0 \, dt\rangle  \\
& = \lim_{t_n \to \infty} \langle f ,  \frac{1}{t_n} \int_0^{t_n} P_{t}^* \delta_0 \, dt\rangle  +  \langle f ,  \frac{1}{t_n} \int_{t_n}^{t_n +r} P_{t}^* \delta_0 \, dt\rangle  - \langle f ,  \frac{1}{t_n} \int_0^{r} P_{t}^* \delta_0 \, dt\rangle  \\
& =\langle f ,\mu\rangle
\end{align*}
where $ \lim_{t_n \to \infty} \langle f ,  \frac{1}{t_n} \int_{t_n}^{t_n +r} P_{t}^* \delta_0 \, dt\rangle  =0$ and  $\lim_{t_n \to \infty} \langle f ,  \frac{1}{t_n} \int_0^{r} P_{t}^* \delta_0 \, dt\rangle =0$, since $\int_{t_n}^{t_n +r} P_{t}^* \delta_0 \, dt  $ and $\int_0^{r} P_{t}^* \delta_0 \, dt $ are  bounded. Therefore, $P_r^* \mu = \mu$. Now the proof of Lemma \ref{Apx B} is finished.
\end{proof}

\begin{Lem}\label{Lem:MeanMeas}
Set 
\begin{align*}
\nu_n = \frac{1}{t_n} \int_0^{t_n} P_t^* \delta_0 \, dt.
\end{align*}
Assume that $X$ is compactly embedded into a Banach space $X_0$ and 
\begin{align*}
\int_{X} f(\|u\|_X) \nu_n (du) \leq C
\end{align*}
for some function $f:\R\to\R$ such that $\lim_{x\to\infty}f(x)=\infty$, and $C$ does not depend on $n$. Then there exists $\mu \in \mathfrak{P} (X)$ such that $\nu_n \rightharpoonup \mu$ weakly on $X_0$.
\end{Lem}

\begin{proof}[Proof of Lemma \ref{Lem:MeanMeas}]
Consider $B_R(X)$, by Markov inequality
\begin{align*}
\nu_n(X \setminus B_R) = \nu_n ( \norm{u}_X > R) \leq \frac{C}{f(R)},
\end{align*}
then $(\nu_n)$ is tight in $X_0$. Now, the Prokhorov theorem (see Theorem \ref{Thm:Prokh}) implies the existence of $\mu$ in $\mathfrak{P}(X_0)$.

Now, let us show that $\mu (X) =1$. It suffices to show that 
\begin{align*}
\int_{X_0} f(\norm{u}_{X}) \mu (du)  \leq  C < \infty
\end{align*}
since, with such property, we obtain
\begin{align*}
\mu(B_R^c) \leq \frac{\int_{X_0} f(\norm{u}_{X}) \mu (du)}{f(R)} \to 0 \quad \text{ as } R \to \infty.
\end{align*}
Now, let $\chi_R$ be a $C^{\infty}$ function on $[0, \infty)$ such that $\chi_R =1 $ on $[0,R]$ and $\chi_R =0$ on $[R+1 , \infty)$.
\begin{align*}
\int_{X_0} f(\norm{u}_{X})  \chi_R (\norm{u}_{X_0})  \nu_n (du) \leq \int_{X_0} f(\norm{u}_{X})   \nu_n (du) \leq C .
\end{align*}
Since $f(\norm{u}_{X})  \chi_R (\norm{u}_{X_0}) $ is bounded continuous  on $X_0$, by Fatou's lemma, we have 
\begin{align*}
\int_{X_0} f(\norm{u}_{X})  \mu (du) \leq C .
\end{align*}
The proof of Lemma \ref{Lem:MeanMeas} is finished. 
\end{proof}

\bibliographystyle{plain}
\bibliography{fnlsproba}
\end{document}